\documentclass[12pt,reqno,oneside]{amsart}
 \usepackage{verbatim,amsmath,amssymb,cite,xspace}

\usepackage{amssymb, latexsym}
\usepackage{amscd}
\usepackage{enumerate}
\usepackage{amsfonts}

\usepackage{epsfig}
\usepackage{color,cite,graphicx} 
   \usepackage{setspace}

\newtheorem{theorem}{Theorem}[section]
\newtheorem{corollary}[theorem]{Corollary}
\newtheorem{lemma}[theorem]{Lemma}

\theoremstyle{remark}
\newtheorem{remark}[theorem]{Remark}

\theoremstyle{definition}

\newtheorem{claim}{{\bf Claim}}[theorem]

\newcommand{\st}[2]{L^{#1}_tL^{#2}_x}

\newcommand{\la}{\langle}
\newcommand{\ra}{\rangle}
\newcommand{\hl}{\dot{H}^1\times L^2}
\newcommand{\vp}{\phi}
\newcommand{\LL}{\mathcal{L}}
\newcommand{\tl}{\tilde{\lambda}}
\newcommand{\tg}{\tilde{\gamma}}

\newcommand{\OR}{\overrightarrow}
\newcommand{\M}{\mathcal{M}}
\newcommand{\HL}{\dot{H}^1\times L^2}

\newcommand{\RSTE}{L^{6,2}_xL^{\infty}_t}
\newcommand{\RSTEn}{L^{\infty}_xL^2_t}
\newcommand{\DRST}{L^{6/5,2}_xL^{\infty}_t\cap L^{3/2,1}_xL^{2}_t}

\def\R{\mathbb{R}}

\renewcommand{\R}{\mathbb{R}}
\newcommand{\EQ}[1]{\begin{equation}  \begin{split} #1 \end{split} \end{equation} }

\numberwithin{equation}{section}

\begin{document}

\title[]{ Global center stable manifold for the defocusing energy critical  wave equation  with potential}
\author{Hao Jia, Baoping Liu, Wilhelm Schlag, Guixiang Xu}

\begin{abstract}
In this paper we consider the defocusing energy critical wave
equation with a trapping potential in dimension $3$. We prove that the set of initial data for which
solutions scatter to an unstable excited state $(\phi, 0)$ forms a finite
co-dimensional path connected $C^1$ manifold in the energy space. This manifold is a global and unique center-stable manifold associated with $(\phi,0)$. It is constructed in a first step locally around {\em any  solution scattering to} $\phi$, which might be very far away from $\phi$ in the $\dot H^1\times L^2(\R^3)$ norm. In a second crucial step 
 a no-return property is proved for any solution which starts near, but not on the local manifolds. This ensures that the local manifolds form a global one. 
Scattering to an unstable steady state is therefore a non-generic behavior, in a strong topological sense in the energy space. 
This extends our previous result~\cite{JiaLiuSchlagXu} to the nonradial case.  The new ingredients here  are (i) application of  the reversed Strichartz estimate from~\cite{Beceanu2} to construct a local center stable manifold near any solution that scatters to $(\phi, 0)$. This is needed since the endpoint of the standard Strichartz estimates fails nonradially.  (ii) The nonradial channel of energy estimate introduced by Duyckaerts-Kenig-Merle~\cite{DKM3}, 
which is used to  show that solutions that start off but near the local manifolds associated with $\phi$  emit some amount of energy into the far field in excess of the amount of energy beyond that of the steady state~$\phi$.  
\end{abstract}

\address{Hao Jia:  School of Mathematics,
Institute of Advanced study,
1 Einstein Drive, Princeton, New Jersey 08540, U.S.A. }
\email{jiahao@math.ias.edu}
\address{  Baoping Liu: Beijing International Center for Mathematical Research, Peking University, Beijing, China}
\email{baoping@math.pku.edu.cn}
\address{Wilhelm Schlag: University of Chicago, Department of Mathematics, 5734 South University Avenue, Chicago, IL 60636, U.S.A.}
\email{schlag@math.uchicago.edu}
\address{Guixiang Xu: Institute of Applied Physics and Computational Mathematics, Beijing, China}
\email{xu\_guixiang@iapcm.ac.cn }

\thanks{The first author was partially supported by NSF grant DMS-1600779, and grant DMS-1128155 through IAS.  The second author was supported by the NSF of China (No.~11601017) and a startup grant from Peking University.
The third author was partially supported by the NSF, DMS-1500696.
The fourth author was partially supported by the NSF of China (No.~11671046
).}

\maketitle

\section{Introduction}

Fix $\beta>2$. Define
\begin{equation*} 
Y:=\left\{V\in C(\R^3):\,  \sup_{x\in
\R^3}(1+|x|)^{\beta}|V(x)|<\infty\right\}.
\end{equation*}
We study solutions to
\begin{equation}\label{eq:mainequation}
\partial_{tt}u-\Delta u-Vu+u^5=0,
\end{equation}
with initial data $\OR{u}(0)=(u_0,u_1)\in \HL(\R^3)$. Since for a
short time the term $Vu$ can be considered as a small perturbation,
by adaptations of results in \cite{BaGe,Gri1,Gri2,Struwe} we know
for any initial data $(u_0,u_1)\in \dot{H}^1\times L^2(\R^3)$, there
exists a unique solution $${u}(t)\in C([0,\infty),\dot{H}^1)\cap
L^5_tL^{10}_x([0,T)\times \R^3)$$ for any $T<\infty$ to equation
(\ref{eq:mainequation}). Moreover, the energy
\begin{equation*}
\mathcal{E}(\overrightarrow{u}(t))=\mathcal{E}(u(t),\partial_tu(t)):=\int_{\R^3}\Big[\frac{|\nabla
u|^2}{2}+\frac{(\partial_tu)^2}{2}-\frac{Vu^2}{2}+\frac{u^6}{6}\Big](x,t)\,dx
\end{equation*}
is constant for all time.

If $V^+(x) =\max{(V(x),0)}$ is large enough,  then the operator $-\Delta -V$ may have negative eigenvalues. In this case,  the equation admits a unique nontrivial ground state $Q>0$ which is the global minimizer of
\begin{equation*}
J(\phi):=\int_{\R^3}\Big[\frac{|\nabla
\phi|^2}{2}-\frac{V\phi^2}{2}+\frac{\phi^6}{6}\Big]\, dx.
\end{equation*}
In
addition to the ground states $Q$ and $-Q$, there can be a number of
``excited states"  with higher energies (see Appendix~A of
\cite{JiaLiuXu}), which are changing sign steady states to equation
(\ref{eq:mainequation}) and decay like $\frac{c}{(1+|x|)}$. Small excited states are always unstable, but large excited states may be stable. These steady states play a fundamental role in understanding the long time dynamics for finite energy solutions to equation (\ref{eq:mainequation}) with  initial data of arbitrary energy. In the radial case we proved in~\cite{JiaLiuXu,JiaLiuSchlagXu} that if we consider generic radial potential $V\in Y$ such that the equation
 admits only finitely many steady states, which are all hyperbolic~\footnote{This means that  the linearized operator $\LL_{\phi}:=\Delta -V +5\phi^4$ around any steady state $(\phi, 0)$ has no zero eigenvalue nor zero resonance}, then generic data will lead to solutions that scatter to one of the stable steady states~\footnote{We say a steady state $(\phi,0)$ is stable if the linearized operator $\LL_{\phi}$ has no negative eigenvalue.}, while each unstable steady state will attract a finite codimensional $C^1$ manifold in the energy space. The result in \cite{JiaLiuSchlagXu} satisfactorily characterized the global dynamical behavior of all finite energy solutions to equation (\ref{eq:mainequation}) in the radial case.

 The proof in~\cite{JiaLiuXu,JiaLiuSchlagXu} relies crucially on the channel of energy estimate for the linear wave equation which was first developed by Duyckaerts-Kenig-Merle~\cite{DKM1, DKM}. The channel of energy estimate works best for wave equation in dimension 3 with radial data. In this case for many nonlinear problems, it  characterizes the steady states as the only solutions that do not radiate energy in either time direction.
It is an essential ingredient in the work of Duyckaerts, Kenig and Merle~\cite{DKM} where  they established the ``soliton resolution" for all type II solutions (i.e. solutions that stay bounded in energy norm up to time infinity or finite blow up time.) for focusing energy critical wave equation with radial data in $\R^3$.  In the nonradial case or other dimensions, there are only weaker versions of the channel of energy estimate available~\cite{DKMnonradial,RKS,KLLS}, and they have been used to establish
  similar resolution results for focusing energy critical wave equations either under size restriction for the initial data~\cite{DKMnonradial}, or along a sequence of times~\cite{CKLS, JiaKenig, Casey, DJKM}. All the results mentioned here belong to a larger effort that aims to understand the long time dynamics for solutions of dispersive equations in the presence of nontrivial coherent structures~\footnote{As defined in~\cite{SofferICM}, coherent structures are solutions that are localized in space, uniformly in time.
 Examples are solitons, kinks, vortices, monopoles, breathers, etc.}. Due to the limitation of techniques to deal with problems beyond the perturbative regime,    we are still at an early stage of understanding of this type of problem. Hence the current interest is to work    on carefully chosen models in order to develop our intuition and technique.
 
We refer the reader to~\cite{Cotemap1, Cotemap2, cotesoliton, KLS, JiaKenig, KLIS2} and references therein for the related results on equivariant wave maps, and to~\cite{TaoCompact,TsaiYau} for  results on nonlinear Schr\"{o}dinger equations with potential.

In this paper, we consider nonradial solutions to (\ref{eq:mainequation}) and
construct the global center stable manifold for unstable excited states.  This gives us a better understanding  of the non-generic behavior of solutions. More precisely, our result shows that solutions that scatter to unstable excited states form a finite co-dimentional manifold in the energy space and hence such solutions are non-generic in a very precise, topological sense.  Although such results are expected, it is often not easy to rigorously confirm them, in a non-perturbative setting. 

More precisely, we say a solution $\OR{u}$ scatters to  steady state $(\phi,0)$  as $t\rightarrow +\infty$ if  there exists a finite energy free wave $\OR{u}^L$ (solution to the linear wave equation) such that
\[\|\OR{u}(t) -(\phi,0) -\OR{u}^L(t)\|_{\HL}\rightarrow 0, \hspace{1cm} \text{ as  } t\rightarrow +\infty.\]
We establish the following result.

 \begin{theorem}\label{th:maintheoremintro}
Let $\Omega$ be an open dense subset of $Y$ such that equation
(\ref{eq:mainequation}) with $V\in \Omega$ has only finitely many steady states which are all hyperbolic.\footnote{We know that such $\Omega$ exists by easy adaptions of arguments in \cite{JiaLiuXu}.} Let $\Sigma$ be the set of    steady states.
 Denote $\OR{u}(t):=\OR{S}(t)(u_0,u_1)$ as the solution to equation (\ref{eq:mainequation}) with  initial data $(u_0,u_1)\in\HL(\R^3)$.
 For each $(\phi,0)\in\Sigma$, define
\begin{equation}
\M_{\phi}:=\left\{(u_0,u_1)\in
\HL(\R^3):\,\OR{S}(t)(u_0,u_1)\,\,{\rm
scatters\,\,to\,\,}(\phi,0)\,{\rm as}\,\,t\to +\infty\right\}.
\end{equation}
Denote
\begin{equation}
\mathcal{L}_{\phi}:=-\Delta -V+5\phi^4
\end{equation}
as the linearized operator around $\phi$. If $\mathcal{L}_{\phi}$
has no negative eigenvalues, then $\M_{\phi}$ is an open set
$\subseteq \HL(\R^3)$. If $\mathcal{L}_{\phi}$  has $n$ negative eigenvalues, then $\M_{\phi}$ is a path
connected $C^1$ manifold $\subset \HL(\R^3)$ of co-dimension $n$.
\end{theorem}
 
We note that there is no smallness assumption in the theorem, and the manifold can extend arbitrarily far away from the unstable steady state relative to the norm in $\dot H^1\times L^2(\R^3)$. 

Theorem \ref{th:maintheoremintro} characterizes all solutions that scatter to a steady state. We expect that generically all solutions scatter to steady states. In the radial case, it was proved that for generic potential all finite energy solutions scatter to one of the steady states, but the proof depends on a particular form of the channel of energy inequality which is valid only in three dimensions and in the radial case. In the nonradial case, it remains an open problem how to characterize the generic behavior. It is perhaps worth pointing out that all nonradial large data results in the study of dynamics of nonlinear dispersive equations depend crucially on monotonicity formulae which are sensitive to algebraic features of the equation. There are currently no effective monotonicity formulae known for equation (\ref{eq:mainequation}) in the nonradial case.

Compared with the radial case~\cite{JiaLiuSchlagXu}, we have two main difficulties in constructing the manifold: 

 (i) Consider any solution $\OR{U}$ that scatters to unstable excited states $(\phi,0)$. When we perturb around $U$, i.e., we write the solution as $U+\eta$, the resulting nonlinearity contains quadratic terms like $U(t)^3\eta^2$ which have a component that behaves like $\phi^3\eta^2$. Standard Strichartz estimate requires control of the nonlinearity in spaces such as $L^1_tL^2_x$, which forces us to estimate $\eta$ in the endpoint Strichartz norm $L^2_tL^\infty_x$.
However,   the endpoint Strichartz estimate for free waves was shown to be false for general   data in~\cite{KM}. To overcome this technical obstacle, we use the    reversed Strichartz estimate due to  Beceanu and Goldberg~\cite{Beceanu2}.  By reverse Strichartz estimates, we mean estimates in the space $\|\cdot\|_{L^p_xL^q_t}$. That is, we first integrate in time and then in space, which is the reverse order of integration for the usual Strichartz estimates. This order of integration arises naturally  in the context of KdV and derivative nonlinear  Schr\"{o}dinger equations, where the local smoothing effect needs to be exploited. For the wave equation the advantages of space-time reversal are less well-known, see however Proposition~3.1 in~\cite{KriNakSch} for an example of an $L^\infty_x L^1_t$ estimate which fails for~$L^1_tL^\infty_x $. 
 In that reference as well as in our case, the main feature is that the fundamental solution for linear wave equation in three dimensions is nonnegative and is integrable in time:
\EQ{\label{eq:Fund}
\int_{0}^{\infty}\frac{1}{t}\delta(|x|-t)\,dt=\frac{1}{|x|}.
}
This property can be used to trade decay in space for decay in time.  For the  $\phi^3\eta^2$ term, which is only quadratic in $\eta$,  there is not enough decay in time to use the standard Strichartz estimates. On the other hand, there is enough decay in space thanks to the $\phi^3$ term. This is exactly the right kind of problem where the reverse Strichartz estimates are more effective.  

Using the reverse Strichartz estimates, we can follow the same techniques in \cite{JiaLiuSchlagXu} to construct a local, finite co-dimensional center stable manifold $\mathcal{M}$ near $\OR{U}(0)$ with the property that if a solution $u$ starts on the manifold, i.e., $\OR{u}(0)\in\mathcal{M}$, then $\OR{u}(t)$ stays close to $\OR{U}(t)$ for all $t\ge 0$ and scatters to $(\phi,\,0)$ as $t\to\infty$; if on the other hand, $\OR{u}(0)$ is close to $\OR{U}(0)$ but not on the manifold, then no matter how small $\left\|\OR{u}(0)-\OR{U}(0)\right\|_{\HL}$ is, $\OR{u}(t)$ will deviate from $\OR{U}(t)$ by a fixed amount at a future time.  

 (ii) 
 The   local manifold construction ensures that any  solution $\OR{u}(t)$ starting off the local manifold, i.e.,   $\OR{u}(0)\in B_{\epsilon}(\OR{U}(0))\backslash \mathcal{M}_\phi$, will leave the time dependent neighborhood $B_{\epsilon}(\OR{U}(t))$ eventually. Up to this point, the argument is still essentially based on perturbative techniques. However, perturbative arguments alone are not sufficient to determine the dynamics when $\OR{u}(t)$ and $\OR{U}(t)$ separate from each other. In order to obtain information on the dynamics for all times, we use the channel of energy inequality introduced by Duyckaerts-Kenig-Merle \cite{DKM3} to show that the solution $u$ necessarily radiates energy into the far field after it leaves $B_{\epsilon}(\OR{U}(t))$.  This is the crucial global component in our paper.  The channel of energy inequality we use here works for nonradial solutions and is not sensitive to the dimension. For another channel of energy inequality which applies in the nonradial case and in all dimensions, see the one for outgoing waves in \cite{DJKM} and \cite{DJKMmap}.  More precisely,
 since $\OR{U}$ scatters to $\phi$, at large times  we know that $\OR{U}(t)$ can essentially be identified as a free radiation at large distances and  $(\phi,0)$ in the finite region. 
 If we take initial data $\OR{u}(0)$ and $\OR{U}(0)$ close enough so that at a given large time $t$ the solutions are still sufficiently close,  we can conclude that locally $\OR{u}(t)$ is essentially $(\phi,0)$ plus a small but nontrivial perturbation.  We will show that the perturbation contains a nontrivial unstable mode, which grows exponentially. Hence at a later time, when the unstable mode dominates all other modes, we use the channel of energy estimate in Lemma~\ref{nonlinear} to conclude that  $\OR{u}$ will send out a fixed amount of energy into large distances and hence the energy left  in the finite region is  strictly less than that of $(\phi, 0)$. From this we know that $\OR{u}$ cannot scatter to $(\phi,0)$. It is interesting to note that our argument shows that a solution, which starts close, but off of the manifold and far away from the unstable steady state, exhibits {\it two} types of radiation: a first   radiation so that locally in space it is close to the unstable steady state at   large times,  and a second radiation which eventually pulls it off the steady state forever.  

In effect, this second step is in the nature of a {\em one-pass theorem}, see~\cite{NakSch1, NakSch2, NakSch3}. While a virial identity is the key for the one-pass theorem in those references, here it is an exterior energy estimate. 
 
Our paper is organized as follows.  
   In Section~\ref{sec:localmanifold},  we construct the local center stable manifold for each solution that scatters to $\phi$.  In Section~\ref{sec:4} we recall the perturbation lemma,  prove the channel of energy estimate and also prove a result on the growth of the unstable modes. Lastly,    in Section~\ref{sec:5} we prove our main result Theorem~\ref{th:maintheoremintro}.

\section{Construction of the local center-stable  manifold}
\label{sec:localmanifold}

We begin with some notation. We use   $c,C>0$  to denote   positive  constants that may be different from line to line.  For nonnegative quantities $X$ and $Y$, we write $X\lesssim Y$ when $X\leq CY$ for some non-essential $C>0$.  When a given operator $L$ has negative eigenvalues, we denote these  as $-k^2$ with $k>0$. Since we work with fixed potentials, we allow all constants to depend on the potential.

Let us first recall the definition of Lorentz spaces $L_x^{p,q}(\R^3)$ for $0<p<\infty$ and $0<q\leq \infty$
\[\|f\|_{L^{p,q}_x(\R^3)}:=p^{\frac{1}{q}}\|\lambda \mu\{|f|\geq \lambda\}^{\frac{1}{p}}\|_{L^q(\R^+, \frac{d\lambda}{\lambda})}.\]
Here  $\mu$ is the standard Lebesgue measure on $\R^3$. Clearly $L^{p,p}=L^p$ for any $0<p<\infty$. We adopt the usual convention that $L^{\infty, \infty}=L^\infty$. Notice that $L^{p,q}\subset L^{p,r} $ whenever $q<r.$
The H\"{o}lder inequality still holds for Lorentz spaces~\cite{Neil}, viz. 
 \EQ{\label{eq:Ho1}
\| fg\|_{L^{r,s}} &\le r' \|f\|_{L^{p_1,q_1}} \|g\|_{L^{p_2,q_2}}  \text{\ \ provided}\\
 \frac{1}{p_1}+\frac{1}{p_2} &= \frac1r<1,\; \frac{1}{q_1}+\frac{1}{q_2}\ge\frac1s
 }
 and the endpoint
 \EQ{\label{eq:Ho2}
\| fg\|_{L^{1}} &\le \|f\|_{L^{p,q_1}} \|g\|_{L^{p',q_2}}, \quad \frac{1}{q_1}+\frac{1}{q_2}\ge 1.
}
  Young's inequalities read as follows: 
\EQ{\label{eq:Y1}
\| f\ast g\|_{L^{r,s}} &\le 3r \|f\|_{L^{p_1,q_1}} \|g\|_{L^{p_2,q_2}} \text{\ \ provided} \\
\frac{1}{p_1}+\frac{1}{p_2} &= \frac1r+1>1,\;\frac{1}{q_1}+\frac{1}{q_2}\ge\frac1s
}
and the endpoint
\EQ{\label{eq:Y2}
\| f\ast g\|_{L^{\infty}} &\le \|f\|_{L^{p,q_1}} \|g\|_{L^{p',q_2}}, \quad \frac{1}{q_1}+\frac{1}{q_2}\ge 1.
}
Since  $Y\subset L^{\frac32, 1}_x(\R^3)$,  Theorem 3, Theorem 1 and Corollary 2 of ~\cite{Beceanu2} imply the following {\em reversed Strichartz estimate} for wave equations with a potential $V\in Y$  in~$\R^3$.

\begin{lemma}\label{lm:RervStrichartzwithPotential}
Take $V\in Y$ such that the operator $-\Delta -V$ has no zero
eigenvalues or zero resonance. Denote by $P^{\perp}$   the projection
operator onto the continuous spectrum of $-\Delta-V$. Let
\begin{equation}\label{eq:sqrtoperator}
\omega:=\sqrt{P^{\perp}(-\Delta -V)}.
\end{equation}
Let $I$ be a time interval with $t_0\in I$. Then for any $(f,g)\in
\HL(\R^3)$ and $F\in \DRST(\R^3\times I)$, the solution
$\OR{\gamma}(t)=(\gamma(t),\,\partial_t\gamma(t))$ to the equation
\begin{equation}
\partial_{tt}\gamma+\omega^2\gamma=P^{\perp}F, \quad (t,x)\in I\times \R^3,
\end{equation}
with $\OR{\gamma}(t_0)=P^{\perp}(f,g)$ satisfies
\begin{equation}\label{eq:strichartzwithpotential}
\|(\gamma,\gamma_t)\|_{C_t^0 (\dot{H}^1\times L^2)}  +
\|\gamma\|_{\RSTE\cap\RSTEn( \R^3\times I)} \leq C
\left(\|(f,g)\|_{\hl} +\|F\|_{\DRST ( \R^3 \times I)}\right).
\end{equation}
\end{lemma}

The appearance of Lorentz spaces here is both natural and essential. Indeed, $|x|^{-1}\in L^{3,\infty}(\R^3)$, and by~\eqref{eq:Ho2} or \eqref{eq:Y2}, 
\[
\sup_{y\in\R^3} \int_{\R^3} |f(x)| |x-y|^{-1}\, dx \le C \|f\|_{L^{\frac32,1}(\R^3)},
\]
cf.~\eqref{eq:Fund}. 
  Our main goal in this section is to prove the following result on the local center stable manifold.
\begin{theorem}\label{th:localmanifold}
Let $\Omega$ be a dense open subset of $Y$ such that equation
(\ref{eq:mainequation}) has only finitely many steady states, all of
which are hyperbolic. Let $V\in \Omega\subset Y$. Suppose that
$\overrightarrow{U}(t)$ is a finite energy solution to equation
(\ref{eq:mainequation}) which scatters to an unstable steady state
$(\phi,0)$. Let
\begin{equation}
-k_1^2\leq-k_2^2\leq\cdots\leq-k_n^2<0
\end{equation}
be the negative eigenvalues  of $\LL_{\phi} = -\Delta-V+5\vp^4$ (counted with
multiplicity) with orthonormal  eigenfunctions
$\rho_1,\,\rho_2,\dots,\rho_n$, respectively. We denote by $P_i$   the projection operator
onto the $i$-th eigenfunction and by $P^{\perp}$
 the projection operator  onto the continuous spectrum, i.e.,
 \begin{align*}
 P_i =\rho_i \otimes \rho_i , \quad\quad
 P^{\perp} =I -\sum_{i=1}^n\rho_i \otimes \rho_i.
 \end{align*} Decompose
\begin{equation}
\HL(\R^3)=X_{cs}\oplus X_u,
\end{equation}
where
\begin{equation}
X_{cs}=\left\{(u_0,u_1)\in \HL(\R^3):\,\langle k_ju_0+
u_1,\rho_j\rangle_{L^2}=0,\,{\rm for\,all}\,1\leq j\leq n\right\},
\end{equation}
and
\begin{equation}
X_u={\rm span}\,\left\{(\rho_j,k_j\rho_j),\,\,1\leq j\leq n\right\}.
\end{equation}
Then there exist $\epsilon_0>0$, $T$ sufficiently large, a ball
$B_{\epsilon_0}((0,0))\subset \HL(\R^3)$, and a smooth mapping
\begin{equation}
\Psi: \overrightarrow{U}(T)+\left(B_{\epsilon_0}((0,0))\cap
X_{cs}\right)\longrightarrow \HL,
\end{equation}
satisfying $\Psi(\overrightarrow{U}(T))=\overrightarrow{U}(T)$, with
the following property. Let $\widetilde{\mathcal{M}}$ be the graph
of $\Psi$ and set $\mathcal{M}=\OR{S}(-T)\widetilde{\mathcal{M}}$, where $\OR{S}(t)$ denotes the solution map for equation (\ref{eq:mainequation}).
Then any solution to equation (\ref{eq:mainequation}) with initial
data $(u_0,u_1)\in\mathcal{M}$ scatters to $(\vp,0)$. Moreover,
there is an $\epsilon_1$ with $0<\epsilon_1<\epsilon_0$, such that
if a solution $\overrightarrow{u}(t)$ with initial data
$(u_0,u_1)\in B_{\epsilon_1}(\overrightarrow{U}(0))\subset
\HL(\R^3)$ satisfies
\begin{equation}
\|\overrightarrow{u}(t)-\overrightarrow{U}(t)\|_{\dot{H}^1\times
L^2}<\epsilon_1 \,\,{\rm for\,\,all\,\,}t\ge 0,
\end{equation}
then $(u_0,u_1)\in \mathcal{M}$.
\end{theorem}

\smallskip
\noindent
{\it Remark:}  $\Omega$ as in the theorem exists, see \cite{JiaLiuSchlagXu}. The proof of Theorem \ref{th:localmanifold}  closely follows the argument for the local manifold in the radial case in \cite{JiaLiuSchlagXu}. However, there is an important additional technical difficulty: in order to control the quadratic nonlinear term $\phi^3\eta^2$ in $\eta$, we need to use reversed Strichartz estimates instead of the endpoint version of the standard Strichartz estimates --- which do not hold in the nonradial case. We note that   if $\epsilon_1$ satisfies the theorem, then any smaller $\epsilon_1$ will also suffice.

\begin{proof}By the assumption that $\OR{U}$ scatters to $\phi$, there exists a free radiation $\OR{U}^L \in \HL(\R^3)$, such that
\begin{equation}\label{scatter-vp}
\lim_{t\to  \infty}\|\OR{U}(t)-(\phi,0)-
\OR{U}^L(t)\|_{\dot{H}^1\times L^2}=0.
\end{equation}
We now divide the construction of the center-stable  manifold into
the following four steps as those in \cite{JiaLiuSchlagXu}.

\smallskip

\noindent {\it \underline{Step 1}: $L^6$ decay for free waves.}  We
observe that for any finite energy free radiation $\OR{U}^L$, we
have
\begin{equation} \|U^L(t)\|_{L^6_x} \rightarrow 0 \quad\quad \text{ as }\quad t\rightarrow \infty. \label{linear-small}\end{equation}
This is a simple consequence of the dispersive estimate for smooth  
free waves, and an approximation argument.

\smallskip

\noindent {\it \underline{Step 2}: reversed space-time estimates for
the radiation term $U-\vp$.} Denote $h(t,x) =U(t,x) -\vp(x)$, then
the radiation term $h$ satisfies
\begin{equation}h_{tt}-\Delta h -V(x)h + 5\vp^4 h +N(\vp, h)=0, \label{eq-h}\end{equation}
where
\begin{equation*}
N(\vp, h) =(\vp+h)^5 -\phi^5 -5\vp^4h= 10 \vp^3h^2 + 10 \vp^2 h^3 +5
\vp h^4 +h^5.
\end{equation*}
 In what follows, we will show that
\[\|h\|_{L^{6,2}_xL^{\infty}_t\cap L^{\infty}_xL^2_t ( \R^3\times[T,\infty))} <\infty,\]
for sufficiently large $T$. 
 From Agmon's estimate in \cite{agmon}, the eigenfunctions $\{\rho_i\}_i$ decay exponentially. Decomposing
\begin{equation*}
h = \lambda_1(t)\rho_1 +\cdots + \lambda_n(t)\rho_n +\gamma,
\end{equation*}
with $\gamma\perp \rho_i$ for $ i=1,\cdots ,n$,  and plugging this
into equation (\ref{eq-h}), we obtain
\begin{align}\label{eq:parametrizedequation}
 \sum_{i=1}^n(\ddot{\lambda}_i(t) -k_i^2 \lambda_i(t))\rho_i + \ddot{\gamma} +\mathcal{L}_{\vp}\gamma = N(\vp,h),
 \end{align}
where $\mathcal{L}_{\vp}=-\Delta-V+5\phi^4$.
By orthogonality between $\gamma(t)$ and $\rho_i, i=1,\ldots, n$, we
derive the following equations for $\lambda_i(t)$ and $\gamma(t,x)$:
 \begin{equation}\label{system}\left\{\begin{aligned}
 \ddot{\lambda}_i(t) -k_i^2 \lambda_i(t) & = P_i N(\vp,h):=N_{\rho_i},\quad\quad i=1,\ldots n\\
 \ddot{\gamma} +\omega^2 \gamma &=P^{\perp} N(\vp,h) := N_c, \quad \quad\omega:=\sqrt{P^{\perp}\mathcal{L}_{\vp}}.
 \end{aligned}\right.\end{equation}

By the decay of the potential $V$ and the steady state $\vp$, we
know that $-V+5\vp^4$ in the linearized operator $\mathcal{L}_{\vp}$
decays like
   $O(\frac{1}{(1+|x|)^{\min\{\beta,4\}}})$, which is better than
the critical rate $O(\frac{1}{|x|^2})$ as $|x|\rightarrow \infty$. Hence we can  apply the result of
Proposition~$6$ in \cite{Beceanu2} and conclude that the reversed
Strichartz estimates as in Lemma~\ref{lm:RervStrichartzwithPotential} hold for solutions
    to the equation
\begin{equation}\label{eq:linearequationwithpotential}
\gamma_{tt}+\omega^2 \gamma =F,
\end{equation}
where $F$ satisfies the compatibility condition $P^{\perp}F=F$.

From (\ref{scatter-vp}) and (\ref{linear-small}), we know that
$$\lim_{T\rightarrow
\infty}\|h(t,x)\|_{L^\infty_tL^6_x([T,\infty)\times\R^3)} =0. $$
Also by the exponential decay of $\rho_i$, we   have
\[|\lambda_i(t)|=|\la \rho_i| h\ra |\leq \|\rho_i\|_{L^{\frac65}} \|h(t,x)\|_{L^6_x(\R^3)}\rightarrow 0 \hspace{1cm}\text{ as }t\rightarrow \infty .\]
Let $\Gamma(t)$ be the solution operator for the  equation
$\gamma_{tt}+\omega^2 \gamma=0$, i.e.,
\[\Gamma(t-t_0)(\gamma(t_0), \dot{\gamma}(t_0)) = \cos (\omega (t-t_0))\gamma(t_0)+\frac{1}{\omega}\sin (\omega (t-t_0))\dot{\gamma}(t_0).\]
We claim:

\begin{claim}\label{claim1}\begin{align}\label{small:linearevolt}
\lim_{T\rightarrow +\infty}\|\Gamma(t-T)(\gamma(T),
\dot{\gamma}(T))\|_{\RSTE\cap\RSTEn( \R^3\times[T,\infty))} =0.
\end{align}
\end{claim}
We postpone the proof of Claim~\ref{claim1} to the end  of this section. 

\medskip

Hence given a small positive number $\epsilon \ll 1$, which will be
chosen later,  we can pick a large time $T=T(\epsilon,U)$, such that
\begin{align}
\|h\|_{\st{\infty}{6}([T,\infty)\times \R^3)}& \leq \epsilon \label{h-small}\\
\|\lambda_i(t)\|_{L^\infty_t ([T,\infty))} &\leq \epsilon \hspace{1cm}    \label{small-lambda} \\
  \|\Gamma(t-T)(\gamma(T), \dot{\gamma}(T))\|_{\RSTE\cap\RSTEn( \R^3\times[T,\infty))} &\leq \epsilon. \label{small-gamma}
\end{align}
From (\ref{small-gamma}) and (\ref{small-lambda}), by the reverse Strichartz estimates in Lemma \ref{lm:RervStrichartzwithPotential}, it follows that the linear solution $h^L$ to $$\partial_{tt}h^L-\Delta h^L +5\phi^4 h^L -Vh^L=0,$$
with initial data $\OR{h}^L(T)=\OR{h}(T)$ satisfies that
\begin{equation}\label{eq:smalllinearpotentialevolution}
\left\|h^L\right\|_{\RSTE\cap\RSTEn( \R^3\times[T,\widetilde{T}))}\leq \frac{K}{2}\epsilon,
\end{equation}
if $\widetilde{T}$ is sufficiently close to $T$.
 We can then use standard perturbation arguments to show that $h\in \RSTE\cap\RSTEn( \R^3\times[T,\widetilde{T}))$ with
$$\|h\|_{\RSTE\cap\RSTEn( \R^3\times[T,\widetilde{T}))}\leq K\epsilon,$$
as long as we choose $\epsilon$ to be sufficiently small. Here we take $K$ large enough so that it dominates any constants appearing in the reverse Strichartz estimates. By a continuity argument, we shall prove that
\begin{equation}\label{continuity-hypo}
\|h\|_{\RSTE\cap\RSTEn( \R^3\times[T,\widetilde{T}))}\leq K\epsilon,
\end{equation}
for all $\widetilde{T}$, not just for $\widetilde{T}$ that are close to $T$.
Suppose that (\ref{continuity-hypo}) holds for $\widetilde{T}$, we shall show that for a small $\delta>0$,  (\ref{continuity-hypo}) holds for $\widetilde{T}+\delta$.

\begin{claim}\label{claim2}
 Let $h$ be a solution to the equation (\ref{eq-h}) with $$\|h\|_{\RSTE\cap\RSTEn( \R^3\times[T,\tilde{T}) )}\leq K\epsilon$$ and
$$\|h\|_{L^{\infty}_tL^6_x(\mathbb{R}^3\times[T,\tilde{T}))}\leq \epsilon.$$
Suppose that $K>10$.
If $\epsilon$ is sufficiently small, then for $\delta>0$ sufficiently small, we have
 \[\|h\|_{\RSTE\cap\RSTEn( \R^3\times[T,\tilde{T}+\delta) )}\leq C_1K\epsilon,\]
where $C_1$ is a constant that only depends on $V$.
 \end{claim}

The claim will be proved at the end of the theorem. We note that due to the use of $L^{6,2}_xL^{\infty}_t$ type spaces, the continuity in time is not obvious. From the equation for $\lambda_i(t)$ in (\ref{system}) and the uniform bound (\ref{small-lambda}) on $\lambda_i$, we conclude
that for $t\geq T$
\begin{align*}
\lambda_i(t)= & \cosh (k_i(t-T))\lambda_i(T)+\frac{1}{k_i}\sinh(k_i(t-T)) \dot{\lambda}_i(T)\\ &+\frac{1}{k_i}\int_T^t \sinh (k_i(t-s))N_{\rho_i}(s)\, ds\\
= &
\frac{e^{k_i(t-T)}}{2}\left[\lambda_i(T)+\frac{1}{k_i}\dot{\lambda}_i(T)
+\frac{1}{k_i}\int_T^t e^{k_i(T-s)}N_{\rho_i}(s)\,
ds\right]+\mathcal{R}(t),
\end{align*}
where $\mathcal{R}(t)$ denotes a term  that remains bounded for
bounded $N_{\rho_i}(s)$. By (\ref{small-lambda}) and the above formula,
 we obtain the following stability condition
\begin{equation}\label{condition} \dot{\lambda}_i(T)= - k_i\lambda_i(T)- \int_T^\infty e^{k_i(T-s)}N_{\rho_i}(s)ds. \end{equation}
Under this condition we can rewrite equation (\ref{system}) as the
following integral equation
\begin{equation}\label{system-1}\left\{\begin{aligned}
\lambda_i(t)&=e^{-k_i(t-T)}\left[\lambda_i(T)+\frac{1}{2k_i} \int_T^\infty e^{k_i (T-s)}N_{\rho_i}(s)ds\right]-\frac{1}{2k_i}\int_T^\infty e^{-k_i|t-s|}N_{\rho_i}(s)ds \\
&=e^{-k_i(t-T)} \left[\lambda_i(T)+\frac{1}{2k_i}
\int_T^{\widetilde{T}+\delta} e^{k_i (T-s)}N_{\rho_i}(s)ds\right]
-\frac{1}{2k_i} \int_T^{\widetilde{T}+\delta} e^{-k_i|t-s|}N_{\rho_i}(s)ds \\
& \qquad \qquad +  e^{-k_i(t-T)}\frac{1}{2k_i}\int_{\widetilde{T}+\delta}^\infty e^{k_i (T-s)}N_{\rho_i}(s)ds -\frac{1}{2k_i} \int_{\widetilde{T}+\delta}^\infty  e^{-k_i|t-s|}N_{\rho_i}(s)ds,\\
 \gamma(t)&=\cos (\omega (t-T))\gamma(T)+\frac{1}{\omega}\sin
(\omega (t-T))\dot{\gamma}(T)+\int_T^t\frac{\sin
(\omega(t-s))}{\omega}N_c (s)ds.
\end{aligned}\right.\end{equation}
By (\ref{system-1}) and the reversed Strichartz estimates in Lemma
\ref{lm:RervStrichartzwithPotential}, we get that
\begin{align}
\|\lambda_i(t)\|_{L^2([T, \widetilde{T}+\delta))}\leq &C\left(
|\lambda_i(T)| + \|N_{\rho_i}\|_{L^{2}_t([T,\widetilde{T}+\delta))} +
\|N_{\rho_i}\|_{L^{\infty}_t([\widetilde{T}+\delta,\infty))}\right),
\label{lambda} \end{align}
 and
 \begin{align} &\|\gamma\|_{\RSTE\cap
\RSTEn(\R^3\times [T,\widetilde{T}+\delta))} \notag\\
\leq & C
\left(\|\Gamma(t-T)(\gamma(T), \dot{\gamma}(T))\|_{\RSTE\cap \RSTEn(
\R^3\times[T,\widetilde{T}+\delta) )} + \|N_c\|_{\DRST(
\R^3\times[T,\widetilde{T}+\delta) )}\right). \label{gamma}
\end{align}
Here the constant $C$ depends on the $L^1$ and $L^2$ integrals of
$e^{-k_it}$ and on the constants in
the reversed Strichartz estimates\footnote{Instead of estimating the
energy norm $\hl(\R^3)$ of $(\gamma(T), \dot{\gamma}(T))$ in
(\ref{gamma}), which may not be small, we estimate its free
evolution in $\RSTE\cap \RSTEn(\R^3\times [T,\widetilde{T}+\delta))$.
Consequently,  we can obtain smallness thanks to
(\ref{small-gamma}).}.

On the one hand, by the fact that $$N_{\rho_i}=\la \rho_i | N(\vp,h)\ra,
\; N_{\rho}=\sum_{i} N_{\rho_i}\, \rho_{i},\; N_c=N-N_{\rho}$$ and the
exponential decay of $\rho_i$, we have
\begin{equation}\|N_{\rho_i}\|_{L^{2}_t([T,\widetilde{T}+\delta))}+ \|N_c\|_{\DRST(\R^3\times [T,\widetilde{T}+\delta))}  \leq  C \|N(\vp,h)\|_{\DRST(\R^3\times[T,\widetilde{T}+\delta))}. \label{N-bound}\end{equation}
By the H\"older inequality in Lorentz spaces, noting that $\phi$ does not depend on time,
we have
\EQ{ \nonumber 
& \|\vp^3 h^2\|_{\DRST(\R^3\times[T,\widetilde{T}+\delta))} \\
& \lesssim \|\vp\|_{L^{6}_x}^3\left( \|h\|^2_{\RSTE(\R^3\times[T,\widetilde{T}+\delta))} + \|h\|_{\RSTE(\R^3\times[T,\widetilde{T}+\delta))} \|h\|_{\RSTEn(\R^3\times[T,\widetilde{T}+\delta))}\right),\\
& \|\vp^2 h^3\|_{\DRST(\R^3\times[T,\widetilde{T}+\delta))}\\
 & \lesssim \|\vp\|_{L^{6}_x}^2\left( \|h\|^3_{\RSTE(\R^3\times[T,\widetilde{T}+\delta))} + \|h\|^2_{\RSTE(\R^3\times[T,\widetilde{T}+\delta))} \|h\|_{\RSTEn(\R^3\times[T,\widetilde{T}+\delta))}\right),\\
}
as well as 
\EQ{\nonumber
& \|\vp h^4\|_{\DRST(\R^3\times[T,\widetilde{T}+\delta))}\\
 & \lesssim \|\vp\|_{L^{6}_x}\left( \|h\|^4_{\RSTE(\R^3\times[T,\widetilde{T}+\delta))} + \|h\|^3_{\RSTE(\R^3\times[T,\widetilde{T}+\delta))} \|h\|_{\RSTEn(\R^3\times[T,\widetilde{T}+\delta))}\right),\\
& \|h^5\|_{\DRST(\R^3\times[T,\widetilde{T}+\delta))}  \\
& \lesssim  \|h\|^5_{\RSTE(\R^3\times[T,\widetilde{T}+\delta))} +
\|h\|^4_{\RSTE(\R^3\times[T,\widetilde{T}+\delta))}
\|h\|_{\RSTEn(\R^3\times[T,\widetilde{T}+\delta))} . 
}
Consequently
\begin{equation}\label{N-12bound}
\|N_{\rho_i}\|_{L^{2}_t([T,\widetilde{T}+\delta))}+ \|N_c\|_{\DRST(
\R^3\times[T,\widetilde{T}+\delta))} \leq C \sum^{5}_{j=2} \|h\|^j_{\RSTE
\cap \RSTEn (\R^3\times[T,\widetilde{T}+\delta) )}.
\end{equation}
On the other hand, by \eqref{h-small} and the exponential decay of
$\rho_i$, we have
 \begin{align}\|N_{\rho_i}\|_{L^{\infty}_t([\widetilde{T}+\delta,\infty))} & \leq C \|\rho_i\|_{L^{6}_x(\R^3)}\|N(\vp,h)\|_{\st{\infty}{\frac{6}{5}}([\widetilde{T}+\delta,\infty)\times\R^3)}\notag \\ & \leq C \sum_{i=2}^5\|\vp\|_{L^6_x}^{5-i}\|h\|_{\st{\infty}{6}( [\widetilde{T}+\delta,\infty)\times \R^3)}^i \leq C\epsilon^2. \label{N-inftybound}\end{align}
The bounds on $\lambda_i$ and $\gamma$ imply an estimate on $h$ via 
 $$h=\sum_{i}\lambda_i\rho_i+\gamma.$$
 In fact, combining estimates
(\ref{lambda}), (\ref{gamma}), (\ref{N-bound})-(\ref{N-inftybound}),  with (\ref{small-lambda}),
(\ref{small-gamma}) and Claim~\ref{claim2}, 
one concludes that 
 \begin{align*}
  \|h\|_{L^{6,2}_xL^{\infty}\cap L^{\infty}_xL^2_t(\mathbb{R}^3\times[T,\widetilde{T}+\delta))} \leq \frac{K}{2}\epsilon  +C\left\{  \sum^{5}_{j=2}  \left(4C_1K\epsilon +\epsilon\right)^j+  \epsilon^2 \right\},  \end{align*}
here $C$ only depends on the constants in the
reversed Strichartz inequalities and  $\|\vp\|_{L^{6}_x}$ and
$\|\rho_i\|_{L^\infty_x\cap L^{6,2}_x}$.  If we choose $\epsilon\ll 1$,
which can be achieved by taking $T$ sufficiently large,  such that
\[\epsilon +\sum_{j=2}^5 (4C_1K+1)^j \epsilon^{j-1} <1,\]
say,   then it follows that 
 \begin{align}\label{h:small:RSTEcapRSTEn}\|h\|_{\RSTE\cap\RSTEn( \R^3\times[T,\widetilde{T}+\delta) )}\leq K
 \epsilon.
 \end{align}
Hence, by a standard continuity argument, we conclude that (\ref{continuity-hypo}) holds for all $\widetilde{T}>T$ and 
$$h\in L^{6,2}_xL^{\infty}_t\cap L^{\infty}_xL^2_t(\mathbb{R}^3\times[T,\,\infty)).$$

\smallskip

\noindent {\it Step 3: construction of the center-stable manifold
near a solution $U$.}  Given a finite energy solution $U$ to
(\ref{eq:mainequation}) satisfying (\ref{scatter-vp}), we consider
another finite energy solution $u$, with
$\|\OR{U}(T)-\OR{u}(T)\|_{\HL(\R^3)}$ small for a fixed large time
$T$, taken from Step 2. We write $u= U + \eta $, then $\eta$ satisfies
\[\eta_{tt}-\Delta \eta - V(x)\eta +(U+\eta)^5-U^5=0, \,\,(t,x)\in(T,\infty).\]
With $U=\vp +h$, we can rewrite the equation as
\begin{equation}\label{eta-eq}
 \eta_{tt} + \mathcal{L}_{\vp}\eta +\tilde{{N}}(\vp, h, \eta)=0, \,\,(t,x)\in(T,\infty),
\end{equation}
with $$\tilde{{N}}(\vp, h, \eta) = (\vp +h +\eta)^5- (\vp + h)^5 -
5\vp^4\eta.$$
Note that $\tilde{N}$ contains  terms which are linear
in $\eta$. However, a further inspection shows that the coefficients
of the linear terms in $\eta$ contains the factor $h$  and hence
decay in both space and time, and can be made small if we choose $T$
sufficiently large. First decompose $\eta$ as
\begin{equation}\label{eta-formula}\eta =\tl_1(t)\rho_1 +\cdots+ \tl_n(t)\rho_n +\tg, \qquad \tg\perp
\rho_i\end{equation} for $i=1,\cdots, n$. We shall use similar arguments as in step~2
to obtain a solution $\eta$ which stays small for all large,
positive times, with given $(\tl_1(T),\cdots,\tl_n(T))$ and
$(\tg,\dot{\tg})(T)$. Note that in order to determine the solution $\eta$, we still have to determine $\dot{\tl}(T)$.
  We can obtain equations for $\tl_i, \tg$ similar to (\ref{system}). Since we seek  a forward solution which grows at most polynomially, we obtain a similar  necessary and sufficient stability condition as (\ref{condition})
\begin{equation}\label{condition2} \dot{\tl}_i(T)= - k_i\tl_i(T)- \int_T^\infty e^{k_i(T-s)}\tilde{N}_{\rho_i}(s)ds.
\end{equation}
Using equations (\ref{eta-eq}) and (\ref{condition2}) we arrive at
the system of equations for $\tl_i$ and $ \tg$,
\begin{equation}\label{system-2}\left\{\begin{aligned}
\tl_i(t)&=e^{-k_i(t-T)}\left[\tl_i(T)+\frac{1}{2k_i} \int_T^\infty e^{k_i (T-s)}\tilde{N}_{\rho_i}(s)\, ds\right]-\frac{1}{2k_i}\int_T^\infty e^{-k_i|t-s|}\tilde{N}_{\rho_i}(s)\,ds \\
\tg(t)&=\cos (\omega (t-T))\tg(T)+\frac{1}{\omega}\sin (\omega
(t-T))\dot{\tg}(T)+\frac{1}{\omega}\int_T^t\sin
(\omega(t-s))\tilde{N}_c (s)\,ds.
\end{aligned}\right.\end{equation}
Define
\begin{equation} \|(\tl_1,\ldots, \tl_n, \tg)\|_{X}:= \sum_{i=1}^n\|\tl_i(t)\|_{L^\infty_t \cap L_t^2([T,  \infty))} +
 \|\tg\|_{\RSTE \cap \RSTEn (\R^3\times [T,\infty) )}.\label{define-X}\end{equation}
Estimating system (\ref{system-2}), we obtain that
  \begin{align}
\|\tl_i(t)\|_{L^{\infty}\cap L^2([T, \infty))}\lesssim & |\tl_i(T)|
+ \|\tilde{N}_{\rho_i}\|_{L^{\infty}_t\cap L^2_t([T,\infty))} \notag\\
\lesssim & |\tl_i(T)| +  \|\tilde{N}\|_{\DRST( \R^3\times[T,\infty) )}, \label{tilde-lambda} \end{align} and
\begin{align}
\|\tg\|_{\RSTE \cap \RSTEn(\R^3\times [T,\infty))}\lesssim&
\|(\tg(T), \dot{\tg}(T))\|_{\hl} + \|\tilde{N}\|_{\DRST(
\R^3\times[T,\infty) )}. \label{tilde-gamma}
\end{align}
Note that \begin{equation}\displaystyle |\tilde{N}| \lesssim  \sum_{j=1}^4
|\vp^{4-j} h^j \eta |+ \sum_{k\geq 2, i+j+k=5} |\vp^i h^j \eta^k|.\label{tN-formula}\end{equation}
For the linear term in $\eta$,  by the H\"{o}lder
inequalities in  Lorentz spaces \eqref{eq:Ho1}, we get that
\begin{align*}
\|\vp^3 h\eta\|_{\DRST( \R^3\times[T,\infty) )}& \leq  \|\vp\|_{L^6_x}^3 \|h\|_{\RSTE(\R^3\times[T,\infty) )}
\|\eta\|_{\RSTE\cap \RSTEn(\R^3\times[T,\infty) )},\\
\|\vp^2 h^2\eta\|_{\DRST( \R^3\times[T,\infty) )}& \leq
\|\vp\|_{L^6_x}^2 \|h\|^2_{\RSTE(\R^3\times[T,\infty) )}
 \|\eta\|_{\RSTE\cap \RSTEn(\R^3\times[T,\infty) )},\\
\|\vp h^3\eta\|_{\DRST( \R^3\times[T,\infty) )}& \leq
\|\vp\|_{L^6_x}\|h\|^3_{\RSTE(\R^3\times[T,\infty) )}
\|\eta\|_{\RSTE\cap \RSTEn(\R^3\times[T,\infty) )},\\
\| h^4\eta\|_{\DRST( \R^3\times[T,\infty) )}& \leq
\|h\|^4_{\RSTE(\R^3\times[T,\infty) )}\|\eta\|_{\RSTE\cap
\RSTEn(\R^3\times[T,\infty) )}.
\end{align*}
By (\ref{h:small:RSTEcapRSTEn}), we have
\begin{equation}\label{N1-estimate}
\left\|\sum_{j=1}^4 \vp^{4-j} h^j \eta\right\|_{\DRST(
\R^3\times[T,\infty) )}\leq C \epsilon  \|\eta\|_{\RSTE\cap
\RSTEn(\R^3\times[T,\infty) )}.\end{equation} The higher order terms
in $\eta$ are easier to estimate. Similar to the above, we can always estimate $h$ in
$\RSTE$, hence
\begin{align}\label{N2-estimate}
\left\| \sum_{k\geq 2, i+j+k=5}\vp^i h^j \eta^k\right\|_{\DRST(
\R^3\times[T,\infty) )}\leq C \sum_{k=2}^5\|\eta\|^k_{\RSTE \cap
\RSTEn(\R^3\times[T,\infty) )}.
\end{align}

By definition of $X$,
$\|\eta\|_{\RSTE\cap \RSTEn(\R^3\times[T,\infty) )}\leq C
\|(\tl_1, \cdots, \tl_n, \tg)\|_{X}$.
We can combine
\eqref{tilde-lambda}, \eqref{tilde-gamma} and \eqref{N1-estimate},
\eqref{N2-estimate} to get
\begin{align}\|(\tl_1,\cdots, \tl_n, \tg)\|_{X}\leq & \,  L \left(\sum_{i=1}^n |\tl_i(T)| + \|(\tg(T), \dot{\tg}(T))\|_{\hl} \right) \label{XT}
\\ &+ L\epsilon \|(\tl_1,\cdots, \tl_n, \tg)\|_{X} + L \sum_{k=2}^5\|(\tl_1,\cdots, \tl_n, \tg)\|_{X}^k, \notag\end{align}
where $L>1$ is a constant only depending on the constants in the
reversed Strichartz estimates,  $\|\phi\|_{L^6(\R^3)}$ and
$\|\rho_i\|_{L^\infty_x\cap L^{6,2}_x}$ (for convenience of later use, we will also assume $L>n$). This inequality implies
that if we take $\epsilon=\epsilon_0$ sufficiently small (which can
be achieved by choosing $T$ suitably large), 
with
\begin{equation}\label{small-data}\sum_{i=1}^n |\tl_i(T)| + \|(\tg(T), \dot{\tg}(T))\|_{\hl}\leq \epsilon_0,\end{equation}
such that $L^3\epsilon_0<\frac{1}{32}$,
then the map defined by the right-hand side of system
(\ref{system-2}) takes a ball  $B_{2L\epsilon_0}(0)\subseteq X$ into itself. Moreover, we can check by the same argument
that this map is in fact
 a contraction on $B_{2L\epsilon_0}(0)\subseteq X$. Thus for any given small $(\tl_1(T), \cdots \tl_n(T),
\tg(T))$ satisfying (\ref{small-data}), we obtain a unique fixed
point of (\ref{system-2}). It follows that
\begin{equation}
u(t,x):=U(t,x) +\sum_{i=1}^k \tl_i(t)\rho_i+\tg(t,x) \label{u-construction}
\end{equation}
 solves (\ref{eq:mainequation}) on $\R^3\times [T,\infty)$, satisfying
\begin{equation}\|\OR{u}-\OR{U}\|_{L_t^\infty( [T,\infty);\hl)}
\lesssim \sum_{i=1}^n |\tl_i(T)| + \|(\tg(T), \dot{\tg}(T))\|_{\hl} \label{epsilon-difference}\end{equation}
 with Lipschitz dependence on the data $\tl_i(T)$ and $(\tg(T),\dot{\tg}(T))$.
By the smoothness of the nonlinearity $\tilde{N}$, the integral
terms in (\ref{system-2}) depend on $\tl_i, \tg$ smoothly. Hence
$\tl_i(t),\tg(t,x)$ and the solution $u(t,x)$ actually have smooth
dependence on the data.

\smallskip
\noindent {\it Step 4: Proof of scattering.}  In this step, we prove that  the solution $\OR{u}$ constructed in step 3  scatters to the same steady state $(\phi, 0)$ as $\OR{U}$. 

For each solution $\OR{u}$ with the decomposition  (\ref{u-construction}) and any time $T'\geq T$,  we   denote 
\begin{equation} \|(\tl_1,\ldots, \tl_n, \tg)\|_{X[T',\infty)}:= \sum_{i=1}^n\|\tl_i(t)\|_{L^\infty_t \cap L_t^2([T',  \infty))} +
 \|\tg\|_{\RSTE \cap \RSTEn (\R^3\times [T',\infty) )}.\end{equation} 
Here $X[T,\infty)$ is the space $X$ from step~3, and from the construction we know that $$\|(\tl_1,\cdots, \tl_n, \tg)\|_{X[T,\infty)} <2L\epsilon_0 <\frac{1}{16}$$
 We will show that  $\|(\tl_1,\cdots, \tl_n, \tg)\|_{X[T',\infty)} \rightarrow 0$ as  $T'\rightarrow \infty$.

We shall need the following property of the linear evolution, which will be proved towards the end of this section: 
\begin{claim}\label{claim3}  For  $(f_0, f_1)\in P^\perp(\dot{H}^1\times L^2)$,  denote \[f(t,x) = \cos (\omega t)f_0+\frac{1}{\omega}\sin (\omega t
)f_1,\] then we have 
\begin{equation} 
\lim_{T_0\rightarrow \infty}\|f(t,x)\|_{\RSTE\cap\RSTEn( \R^3\times[T_0,\infty) )}= 0. 
\label{largeT}\end{equation}
And there exists a free wave $f^L(t,x)$ with data $\OR{f}^{L}(0)\in \HL$,  such that 
\begin{equation}
\lim_{t\rightarrow +\infty}\|\OR{f}(t,x) - \OR{f}^L(t,x)\|_{\HL} =0. \label{f-scatter}
\end{equation}
\end{claim}

Using (\ref{largeT}) in Claim~\ref{claim3},  for the $\epsilon_0$ chosen in step 3, we can take $T_1>T$ large enough such that 
\begin{align}
\left\|e^{-k_i(t-T)}\tl_i(T)\right\|_{L^\infty_t \cap L_t^2([T_1,  \infty))} & < \epsilon_0^2, \label{T1-lambda}\\ 
\left\|\cos (\omega (t-T))\tg(T)+\frac{1}{\omega}\sin (\omega
(t-T))\dot{\tg}(T)\right\|_{_{\RSTE\cap\RSTEn( \R^3\times[T_1,\infty) )}} & <  \epsilon_0^2. \label{T1-gamma} \end{align}
We control the system (\ref{system-2}) on the interval $[T_1, \infty)$ in the following fashion:  we estimate the linear part on the interval $[T_1,\infty)$ using (\ref{T1-lambda})(\ref{T1-gamma}), and then estimate the nonlinear (integral) term over the larger interval  $[T,\infty)$. This yields 
 \begin{align}
\left\|\tl_i(t)\right\|_{L^{\infty}\cap L^2([T_1, \infty))}\lesssim &\left\|e^{-k_i(t-T)}\tl_i(T)\right\|_{L^\infty_t \cap L_t^2([T_1,  \infty))} 
+\left \|\tilde{N}\right\|_{\DRST( \R^3\times[T,\infty) )}, \label{T1-tilde-lambda}  \\
\|\tg\|_{\RSTE \cap \RSTEn(\R^3\times [T,\infty))}\lesssim&
\left\|\cos (\omega (t-T))\tg(T)+\frac{1}{\omega}\sin (\omega
(t-T))\dot{\tg}(T)\right\|_{{\RSTE\cap\RSTEn( \R^3\times[T_1,\infty) )}}  \notag \\ &+ \|\tilde{N}\|_{\DRST(
\R^3\times[T,\infty) )}. \label{T1-tilde-gamma}
\end{align}
Combing these estimates with (\ref{tN-formula}), (\ref{N1-estimate}), (\ref{N2-estimate}), we infer that (notice we assumed $L>n$)
\begin{align*} & \|(\tl_1,\cdots, \tl_n, \tg)\|_{X[T_1,\infty)}\\
&\leq  \, (n+1) \epsilon_0^2 + L\epsilon_0 \|(\tl_1,\cdots, \tl_n, \tg)\|_{X[T,\infty)} + L \sum_{k=2}^5\|(\tl_1,\cdots, \tl_n, \tg)\|_{X[T,\infty)}^k, \\ 
&\leq  L \epsilon_0^2  + 2L^2\epsilon_0^2 + L \sum_{k=2}^5 (2L\epsilon_0)^k < 2L (2L\epsilon_0)^2.
\end{align*} 
Next,  fix our choice of $T_1$ and rewrite system (\ref{system-2}) by breaking the integral into finer pieces,
\begin{equation}\label{system-4}\left\{\begin{aligned}
\tl_i(t) 
= & e^{-k_i(t-T)}\left[\tl_i(T)+\frac{1}{2k_i} \int_T^{T_1} e^{k_i (T-s)}\tilde{N}_{\rho_i}(s)\, ds\right]-\frac{1}{2k_i}\int_T^{T_1} e^{-k_i|t-s|}\tilde{N}_{\rho_i}(s)\,ds \\
& + e^{-k_i(t-T_1)} \frac{1}{2k_i} \int_{T_1}^\infty e^{k_i (T_1-s)}\tilde{N}_{\rho_i}(s)\, ds -\frac{1}{2k_i}\int_{T_1}^\infty e^{-k_i|t-s|}\tilde{N}_{\rho_i}(s)\,ds, \\
\tg(t)=&\cos (\omega (t-T))\tg(T)+\frac{1}{\omega}\sin (\omega
(t-T))\dot{\tg}(T)+\frac{1}{\omega}\int_T^{T_1}\sin
(\omega(t-s))\tilde{N}_c (s)\,ds \\
 & + \frac{1}{\omega}\int_{T_1}^t\sin
(\omega(t-s))\tilde{N}_c (s)\,ds.
\end{aligned}\right.\end{equation}
We can pick $T_2>T_1$ large enough such that the first line in the expression of $\tl_i$ is small in $L^\infty_t \cap L_t^2([T_2,  \infty))$, 
 also the first line in the expression of $\tg$ is small in ${\RSTE\cap\RSTEn( \R^3\times[T_2,\infty) )}$. We can require that they are bounded by $\epsilon_0^3$.   Note that we used Claim~\ref{claim3} for the term $\frac{1}{\omega}\int_T^{T_1}\sin
(\omega(t-s))\tilde{N}_c (s)\,ds $, which can be viewed as a superposition of linear evolutions.

Then estimating the second line of $\tl_i$ and $\tg$ over the larger interval $[T_1,\infty)$,  we obtain 
\begin{align*} & \|(\tl_1,\cdots, \tl_n, \tg)\|_{X[T_2,\infty)}\\
&\leq  \, (n+1)  \epsilon_0^3 +L\left(  \epsilon_0 \|(\tl_1,\cdots, \tl_n, \tg)\|_{X[T_1,\infty)} +  \sum_{k=2}^5\|(\tl_1,\cdots, \tl_n, \tg)\|_{X[T_1,\infty)}^k \right), \\ 
&\leq    2L(2L\epsilon_0)^3.
\end{align*} 
It is clear that this process can be repeated indefinitely:  once we fix $T_j$, we can rewrite the system  (\ref{system-2}) as in (\ref{system-4}), and find $T_{j+1} >T_j$ such that the first line  is bounded by $\epsilon_0^{j+1}$,  which implies the estimate 
\[\|(\tl_1,\cdots, \tl_n, \tg)\|_{X[T_{j+1},\infty)} \leq  2L(2L\epsilon_0)^{j+1}.\]
In view of (\ref{eta-formula}), (\ref{tN-formula}), (\ref{N1-estimate}), (\ref{N2-estimate}) we conclude that 
\begin{align*}
\lim_{T'\rightarrow +\infty}\|(\tl_1,\cdots, \tl_n, \tg)\|_{X[T',\infty)}  & =0, \\
\lim_{T'\rightarrow \infty}\|\eta\|_{\RSTE\cap\RSTEn( \R^3\times[T',\infty) )} & =   0, \\
\lim_{T'\rightarrow \infty} \|\tilde{N}\|_{\DRST( \R^3\times[T',\infty) )} &= 0.
\end{align*} 
These asymptotics allow us to write  the asymptotic profile of $\tilde{\gamma}$  in the form 
\[\tilde{\gamma}_\infty(t)=\cos (\omega (t-T))\tg(T)+\frac{1}{\omega}\sin (\omega
(t-T))\dot{\tg}(T)+\frac{1}{\omega}\int_T^\infty\sin
(\omega(t-s))\tilde{N}_c (s)\,ds, \] with the property that
\[\|\OR{\tilde{\gamma}} -\OR{\tilde{\gamma}}_\infty\|_{L^\infty_t \dot{H}\times L^2 [T',\infty)}\lesssim   \|\tilde{N}\|_{\DRST( \R^3\times[T',\infty) )}\rightarrow 0 \hspace{1cm} \text{ as }T'\rightarrow +\infty.\]
$\tilde{\gamma}_\infty(t)$ can be further replaced with a free wave by (\ref{f-scatter}) in Claim~\ref{claim3}.  
Combing the preceding with the fact
$\|\tl_i(t)\|_{L^\infty_t \cap L_t^2([T',  \infty))}\rightarrow 0$  as $T'\rightarrow +\infty$, we  conclude that 
$\OR{u}$   scatters to the same steady state  $(\phi, 0)$ as $\OR{U}$.  We can now define
\begin{equation}
\Psi: \overrightarrow{U}(T)+\left(B_{\epsilon_0}((0,0))\cap
X_{cs}\right)\longrightarrow \dot{H}^1\times L^2,
\end{equation}
as follows: for any $(\tg_0,\tg_1)\in
P^{\perp}\left(\HL(\R^3)\right)$   and $\tl_i\in \R$ such that
\begin{equation*}
\mathcal{\xi}:=\sum_{i=1}^n\tl_i(\rho_i,-k_i\rho_i)+(\tg_0,\tg_1)+
\overrightarrow{U}(T)\in
\overrightarrow{U}(T)+\left(B_{\epsilon_0}((0,0))\cap X_{cs}\right),
\end{equation*}
set
\begin{equation*}
\tl_i(T)=\tl_i, \,\,\,{\rm for\,\,}i=1,\dots,n \,\,\,{\rm
and}\,\,(\tg(T),\dot{\tg}(T))=(\tg_0,\tg_1).
\end{equation*}
Then with $\dot{\tl}_i(T)$ given by  (\ref{condition2}), we define
\begin{equation*}
\Psi(\mathcal{\xi}):=\left(\sum_{i=1}^n\tl_i(T)\rho_i+\tg_0,\,\,\sum_{i=1}^n\dot{\tl}_i(T)\rho_i+\tg_1\right)
+\OR{U}(T).
\end{equation*}
If $\epsilon_0$ is chosen sufficiently small, then $\dot{\tl}_i$ is
uniquely determined by contraction mapping in the above. We define
$\widetilde{\mathcal{M}}$ as the graph of $\Psi$ and let
$\mathcal{M}$ be $\OR{S}(-T)(\widetilde{\mathcal{M}})$. We can then
check that $\Psi,\,\mathcal{M},\,\widetilde{\mathcal{M}}$ verify the
requirements of the theorem. Since $\OR{S}(T)$ is a diffeomorphism,
$\mathcal{M}$ is a $C^1$ manifold.

\smallskip

\noindent {\it Step 5: unconditional uniqueness.} Now suppose that a
solution $u$ to equation (\ref{eq:mainequation})  satisfies
 \[\|\OR{u}-\OR{U}\|_{L^\infty( [0,\infty);\hl)}\leq\epsilon_1\ll\epsilon_0.\]
 
 We need to show that $\OR{u}(T)\in \widetilde{\mathcal{M}}$. We denote
 \[\eta(t,x) = u(t,x)-U(t,x) = \sum_{i=1}^n \tl_i(t)\rho_i +\tg(t,x),\]
 then  $\OR{\eta}\in L^\infty_t([0,\infty);\hl)$ with norm smaller than $\epsilon_1$.
Using similar arguments as in {\it Step 2}, we can conclude that for sufficiently large $T$ and $\widetilde{T}$ which is bigger than but close to $T$,
 \begin{align}
& \|\tl_i(t)\|_{L_t^\infty([T,\infty))} +  \|\vec{\tg}(t,x)\|_{ L_t^\infty( [T,\infty);\hl) }\leq C \epsilon_1,\label{delta-small}\\
& \left\|\tl_i(t)\right\|_{L^2([T,\,\widetilde{T}))}\leq C\epsilon_1, \notag \\
&  \left\|\tg\right\|_{ \RSTE\cap \RSTEn( \R^3\times[T,\tilde{T}) )}\leq C\epsilon_1. \notag
 \end{align}
 Notice the $L^\infty$ bound on $\tl_i$ implies that the stability condition   (\ref{condition2}) must hold true,  we are again reduced to (\ref{system-2}).
 Now we wish to show that
 \begin{align}
 \tl_i(t)\in L^2([T,\infty)), \; \tg(t,x)\in \RSTE \cap \RSTEn (\R^3\times[T,\infty) ), \label{uncond unique}
 \end{align}
with a small norm,
which together with the fixed point theorem imply
$\OR{u}(T)\in\widetilde{\mathcal{M}}$. Pulling back from $T$ to $0$,
we can obtain the desired result. To show \eqref{uncond unique}, we follow
similar arguments as in step 2. Define the norm
\begin{align*}
\|(\tl_1,\cdots, \tl_n, \tg)\|_{X([T,\tilde{T}))}:&=
\sum_{i=1}^n\|\tl_i(t)\|_{L_t^2([T, \tilde{T}))} + \|\tg\|_{\RSTE
\cap \RSTEn (\R^3\times[T,\tilde{T}) )}.
\end{align*}
Similar to (\ref{lambda}), (\ref{gamma}), (\ref{N1-estimate}) and
(\ref{N2-estimate}), we get
\begin{align*}
& \sum_{i=1}^n\|\tl_i(t)\|_{L^2([T, \tilde{T}))} +
\|\tg\|_{\RSTE\cap
\RSTEn( \R^3\times[T,\tilde{T}))}\\
\leq &\;C\left(\sum_{i=1}^n |\tl_i(T)| + \|(\tg(T),
\dot{\tg}(T))\|_{\hl} +  \|\tilde{N}\|_{\DRST( \R^3\times[T,\tilde{T}))} +
\|\tilde{N}\|_{L^{\infty}_tL^{\frac{6}{5}}_x([\tilde{T},\infty))}\right)\\
\leq&\; C \epsilon_1+ C \epsilon_0 \|\eta\|_{\RSTE\cap \RSTEn(
\R^3\times[T,\tilde{T}))} +C \sum_{k=2}^5\|\eta\|^k_{\RSTE \cap
\RSTEn(\R^3\times[T,\tilde{T}))} \\
& \quad\;\;\; + C \sum_{i+j+k=5, k\geq
1}\|\vp\|_{L^6_x}^i\|h\|^j_{L^\infty_t L^6_x(
[\tilde{T},\infty))}\|\eta\|^k_{L^\infty_tL^6_x(
[\tilde{T},\infty))} \\
\leq&\; C \epsilon_1+ C \epsilon_0 \|\eta\|_{\RSTE\cap \RSTEn(\R^3\times
[T,\tilde{T}))} +C \sum_{k=2}^5\|\eta\|^k_{\RSTE \cap
\RSTEn(\R^3\times[T,\tilde{T}))},\end{align*} 
where the constant $C$ may change from line
to line. Hence by \eqref{delta-small}, we have
\[\|(\tl_1,\cdots, \tl_n, \tg)\|_{X([T,\tilde{T}))} \leq C \epsilon_1 + C\epsilon_0 \|(\tl_1,\cdots, \tl_n, \tg)\|_{X([T,\tilde{T}))}+L \sum_{k=2}^5\|(\tl_1,\cdots, \tl_n, \tg)\|_{X[T,\tilde{T})}^k.\]
By a continuity argument similar to the one used in {\it Step 2}, we can conclude that
\[\|(\tl_1,\cdots, \tl_n, \tg)\|_{X([T,\infty))} \leq \liminf_{\tilde{T}\rightarrow \infty}\|(\tl_1,\cdots, \tl_n, \tg)\|_{X([T,\tilde{T}))}\leq C \epsilon_1 <\epsilon_0,\]
We omit the routine details.  
 \end{proof}

Now we give the proof for Claim~\ref{claim1}.
Claim~\ref{claim1} will be
proved as a consequence of the following lemma.
\begin{lemma}\label{lem-Gamma}
Let $\overrightarrow{U}^L$ be a  finite energy free radiation
and $(\phi, 0)$ be a steady state to equation
(\ref{eq:mainequation}). Recall that
\[\omega =\sqrt{P^{\perp}(-\Delta - V +5\phi^4)}.\]
Let $\gamma$ be the solution to
\begin{equation}
\left\{\begin{aligned} \partial_{tt}\gamma +\omega^2 \gamma &=0,
\hspace{1cm} \text{ in } [T,\infty)\times \R^3,
\\
\overrightarrow{\gamma}(T) & = P^{\perp} (\overrightarrow{U}^L(T)).
\end{aligned}\right.
\end{equation}
For any $\epsilon>0$, if we take $T=T(\epsilon,
\overrightarrow{U}^L)>0$ sufficiently large, then
\begin{equation}
 \|\gamma\|_{\RSTE
\cap \RSTEn (\R^3\times[T,\infty) )}<\epsilon.
\end{equation}
\end{lemma}
\begin{proof} For a given $\epsilon>0$, fix $0<\delta \ll \epsilon$ to be determined below. We can take a smooth compactly supported (in space) free radiation $\overrightarrow{\widetilde{U}}^L$ such that
\begin{equation}
\|\overrightarrow{U}^L(0)-\overrightarrow{\widetilde{U}}^L(0)\|_{\hl(\R^3)}
\leq \delta.  \label{UU-small}
\end{equation}
Let us assume that ${\rm supp}(\overrightarrow{\widetilde{U}}^L(0))
\Subset B_R(0)$ for some $R>0.$ Hence by the strong Huygens' principle,
for large time $t$ we have
\begin{equation*}
|\widetilde{U}^L(t,x)|\leq \frac{C}{t}\chi_{[t-R\leq|x|\leq t+R]},\,\,\, {\rm
for}\,\, t>R.
\end{equation*}
Now for $T\gg R$, by direct computation we get that
\begin{eqnarray*}
&&\|\widetilde{U}^L(t,x)\|_{L^{6,2}_xL^{\infty}_t(\mathbb{R}^3\times [T,\infty))}\\
&&\lesssim \left\|\frac{1}{t}\chi_{[t-R\leq |x|\leq t+R]}\right\|_{L^{6,2}_xL_t^{\infty}(\mathbb{R}^3\times[T,\infty))}\\
&&\lesssim \left\|\frac{1}{|x|-R}\chi_{[|x|>T-R]}\right\|_{L^{6,2}_x}\\
&&\lesssim \left\|\frac{1}{|x|}\chi_{[|x|>\frac{T}{2}]}\right\|_{L^{6,2}_x}\lesssim \frac{1}{\sqrt{T}}.
\end{eqnarray*}
Similarly,
\begin{eqnarray*}
&&\left\|\widetilde{U}^L(t,x)\right\|_{L^{\infty}_xL^2_t(\mathbb{R}^3\times[T,\infty))}\\
&&\lesssim  \left\|\frac{1}{t}\chi_{[t-R\leq |x|\leq t+R]}\right\|_{L^{\infty}_xL^{2}_t(\mathbb{R}^3\times[T,\infty))}\\
&&\lesssim \left\|\left(\frac{1}{|x|^2}\chi_{[|x|>T-R]}\cdot R\right)^{1/2}\right\|_{L^{\infty}_x}\lesssim_R \frac{1}{T}.
\end{eqnarray*}
Hence   
\begin{equation} \lim_{T\rightarrow \infty}\|\widetilde{U}^L\|_{\RSTE
\cap \RSTEn (\R^3\times[T,\infty) )}=0 \label{small-tildeU}.\end{equation}

Since $\overrightarrow{\widetilde{U}}^L$ is a free
radiation, we see that
\begin{equation}
\partial_{tt} {\widetilde{U}}^L -\Delta  {\widetilde{U}}^L - V  {\widetilde{U}}^L + 5\phi^4  {\widetilde{U}}^L = - V  {\widetilde{U}}^L +5\phi^4  {\widetilde{U}}^L, \quad\text{ in } (0,\infty)\times \R^3. \label{eq:tildeU}
\end{equation}
 By the decay property of $V, \, 5\phi^4$  
 and (\ref{small-tildeU}), simple calculations show that
 \[\lim_{T\rightarrow \infty } \|- V {\widetilde{U}}^L +5\phi^4 {\widetilde{U}}^L\|_{\DRST ( \R^3 \times [T,\infty))} =0.\]
 Choose $T$ sufficiently large, such that
 \begin{align}
 \|\widetilde{U}^L\|_{\RSTE
\cap \RSTEn (\R^3\times[T,\infty) )} \leq & \delta, \label{small-UT}\\  \|- V  {\widetilde{U}}^L +5\phi^4  {\widetilde{U}}^L\|_{\DRST ( \R^3 \times [T,\infty))} \leq & \delta.\label{RHS-small}
 \end{align}
 Note that $\overrightarrow{v}:= \overrightarrow{\gamma} - P^{\perp}\overrightarrow{\widetilde{U}}^L$ solves
 \[\partial_{tt} v +\omega^2\, v = - P^{\perp}\left(  - V  {\widetilde{U}}^L +5\phi^4  {\widetilde{U}}^L \right), \quad (t,x)\in [T, \infty)\times \R^3, \]
 with initial data $\overrightarrow{v}(T)= P^{\perp}\left(
 \overrightarrow{U}^L(T) - P^{\perp}\overrightarrow{\widetilde{U}}^L(T)\right)$. \footnote{By definition, it is clear that $P^{\perp}$ is bounded in $L^{\frac{6}{5},2}_xL^{\infty}_t\cap L^{\frac{3}{2},1}_xL^2_t$.} It is clear from  the bounds (\ref{UU-small}) and  energy conservation for the free radiation that $$\left\|\OR{v}(T)\right\|_{\HL}\leq C\delta.$$
By (\ref{RHS-small})  and reversed Strichartz estimates from Lemma \ref{lm:RervStrichartzwithPotential}, we can conclude that
 \begin{equation}\|v\|_{\RSTE
\cap \RSTEn (\R^3\times[T,\infty) )}\leq C\delta. \label{small-v}\end{equation}
 Combining bounds (\ref{small-v}) and (\ref{small-UT}), and fixing $\delta$ small, the lemma is proved.
\end{proof}

Now the proof of Claim~\ref{claim1} is easy. Note that due to the
fact that
\[\lim_{T\rightarrow \infty} \|\overrightarrow{U}(T) - (\phi,0) -\overrightarrow{U}^L(T)\|_{\hl(\R^3)} =0,\]
we see that the initial data for $\gamma$ satisfies
\[\lim_{T\rightarrow \infty} \|\overrightarrow{\gamma}(T)-P^{\perp}\overrightarrow{U}^L(T)\|_{\hl(\R^3)} =0.\]
Hence  Claim~\ref{claim1} follows from the above lemma and reversed Strichartz
estimates. 

\medskip

\noindent
{\it Proof of Claim \ref{claim2}}: From the bound
$$\|h\|_{\RSTE\cap \RSTEn (\R^3\times[T,\widetilde{T}) )}\leq K\epsilon,$$
we check as in the proof of Theorem \ref{th:localmanifold} that
\begin{equation}\label{eq:estimateonf11}
\|f\|_{L^{\frac{6}{5},2}_xL^{\infty}_t\cap L^{\frac{3}{2},1}_xL^{2}_t(\R^3\times[T,\widetilde{T}))}\lesssim K^5\epsilon^2,
\end{equation}
where $f=N(\phi,h)$. $h$ satisfies
$$\partial_{tt}h-\Delta h-Vh+5\phi^4h+f=0,$$
and thus $\tilde{h}:=P^{\perp}h$ satisfies
$$\partial_{tt}\tilde{h}-\Delta \tilde{h}-V\tilde{h}+5\phi^4\tilde{h}+P^{\perp}f=0.$$
By reverse Strichartz estimates and the estimates (\ref{eq:estimateonf11}) on $f$, we conclude that the solution $\tilde{h}^L$ to
$$\partial_{tt}\tilde{h}^L-\Delta \tilde{h}^L-V\tilde{h}^L+5\phi^4\tilde{h}^L=0$$
with $\OR{\tilde{h}}^L(\widetilde{T})=P^{\perp}(\OR{h}(\widetilde{T}))$ satisfies that
$$\left\|\tilde{h}^L-\tilde{h}\right\|_{\RSTE\cap \RSTEn (\R^3\times[T,\widetilde{T})) }\leq CK^5\epsilon^2,$$
and hence
$$\left\|\tilde{h}^L\right\|_{\RSTE\cap \RSTEn (\R^3\times[T,\widetilde{T})) }\leq C_0K\epsilon+CK^5\epsilon^2.$$
Using approximation by smooth and compactly supported data, it is easy to show that there exists sufficiently small $\delta>0$ such that
$$\left\|\tilde{h}^L\right\|_{\RSTE\cap \RSTEn (\R^3\times[T,\widetilde{T}+\delta)) }\leq C_0K\epsilon+2CK^5\epsilon^2.$$
Hence, by taking $\delta$ smaller if necessary so that the growth of the unstable modes can be controlled, we can conclude that the solution $h^L$ to
$$\partial_{tt}h^L-\Delta h^L-Vh^L+5\phi^4h^L=0$$
with $\OR{h}^L(\widetilde{T})=\OR{h}(\widetilde{T})$ satisfies that
$$\left\|h^L\right\|_{\RSTE\cap \RSTEn (\R^3\times[\widetilde{T},\widetilde{T}+\delta)) }\leq C_0K\epsilon+\epsilon+4CK^5\epsilon^2.$$
Then by a standard perturbation argument, we see that if $\epsilon$ is sufficiently small, then
$$\left\|h\right\|_{\RSTE\cap \RSTEn (\R^3\times[\widetilde{T},\widetilde{T}+\delta)) }\leq C_0K\epsilon+\epsilon+8CK^5\epsilon^2.$$
Combining the above with estimates of $h$ on the interval $[T,\widetilde{T})$ and choosing $C_1\gg C_0$, the claim is proved.

\medskip

\noindent
{\it Proof of Claim \ref{claim3}}:  From  the proof    of Lemma~\ref{lem-Gamma}, we know that for free wave $U^L$ with smooth compactly supported data, we have
\begin{equation}\label{free62}
\lim_{T_0\rightarrow \infty} \|U^L\|_{\RSTE
\cap \RSTEn (\R^3\times[T_0,\infty) )}=0.
\end{equation}
Then by approximation, (\ref{free62}) holds true for any free wave with finite energy.   

Now let  $f(t,x)$ be a solution to the equation (recall $\omega^2 =P^{\perp}(-\Delta -V+5\phi^4)$)
\begin{equation}
\left\{\begin{aligned} \partial_{tt}f +\omega^2 f &=0,
\hspace{2cm} \text{ in } [0,\infty)\times \R^3,
\\
\overrightarrow{u}(0) & = (f_0, f_1)\in P^\perp(\dot{H}^1\times L^2) .\end{aligned}\right.
\end{equation}
For any given $\epsilon>0$, we first take smooth  and compactly supported  data $(\tilde{f}_0, \tilde{f}_1)$ such that \[\|(f_0, f_1) - (\tilde{f}_0, \tilde{f}_1)\|_{\HL} \lesssim \epsilon,\]
which further implies 
\[\|(f_0, f_1) - P^{\perp} (\tilde{f}_0, \tilde{f}_1)\|_{\HL} \leq    \|(f_0, f_1) - (\tilde{f}_0, \tilde{f}_1)\|_{\HL} \lesssim \epsilon.\]

We take $ g(t,x)$ to be the solution to the equation
\begin{equation}\label{ttheta}
\left\{\begin{aligned} \partial_{tt}g +\omega^2 g &=0,
\hspace{2cm} \text{ in } [0,\infty)\times \R^3,
\\
\OR{g}(0) & =  (g_0, g_1) : = P^{\perp} (\tilde{f}_0, \tilde{f}_1).\end{aligned}\right.
\end{equation}
From Strichartz estimates, we have  $g\in \RSTE
\cap \RSTEn (\R^3\times[0,\infty) ) $.

Let us recall  an  estimate from the proof of~\cite[Corollary 2]{Beceanu2}  (page 27 in the journal version)  which   is slightly stronger than the estimate stated in the main result~\cite[Corollary 2]{Beceanu2}.  Notice that it in fact follows from interpolation between the bounds in~\cite[Theorem 1]{Beceanu2}.
For  $0\leq \theta_1, \theta_2 \leq 1  $ and $\theta_1+\theta_2\leq 1, $
\begin{equation} \left\|t^{1-\theta_1-\theta_2} \left(\cos (t\omega)g_0  +  \frac{\sin t\omega}{\omega} g_1\right)\right\|_{(\mathcal{K}^{\theta_2})^*_xL^\infty_t} \lesssim \|\Delta g_0\|_{\mathcal{K}^{\theta_1}} +\|\nabla g_1\|_{\mathcal{K}^{\theta_1}}
 \label{Interpolation}\end{equation}
 It is not necessary for us to give the detailed definition of $\mathcal{K}^{\theta}$ and $(\mathcal{K}^\theta)^*$,  as we only need the embedding property 
 \[L^{\frac{3}{3-\theta},1}\subset \mathcal{K}^{\theta},  \hspace{1cm}  (\mathcal{K}^\theta)^* \subset L^{3/\theta,\infty}.\]
Hence we can take $ \theta_2=\frac12$ and   $\theta_1=0,  $  and  obtain the estimate we need, viz. 
\begin{equation} \| \cos (t\omega)g_0  +  \frac{\sin t\omega}{\omega} g_1\|_{L^{6,\infty}_x L^\infty_t[T_0,\infty)}\lesssim T_0^{-\frac12}\left(\|\Delta g_0\|_{L^{1}} +\|\nabla g_1\|_{L^{1}}\right)
 \label{Interpolation-2}\end{equation}
 Notice that eigenfunctions $\rho_i$ to $\LL_{\phi}=-\Delta -V +5\phi^4$ decay exponentially and $\rho_i\in W^{2,p}, 1\leq p \leq \infty$.  
 Together with the fact $(\tilde{f}_0, \tilde{f}_1)$  is  smooth and compactly supported and $ (g_0, g_1)  = P^{\perp} (\tilde{f}_0, \tilde{f}_1)$,  
 we have $\Delta g_0, \nabla g_1\in {L^{1}}.$

Define the matrix operator 
\[J(t) =\begin{bmatrix} \cos (t|\nabla|)& |\nabla|^{-1}\sin (t|\nabla|)\\  -|\nabla| \sin(t|\nabla|) & \cos(t|\nabla|)\end{bmatrix},\]
and consider the free wave $g^L(t,x)$  with   the initial data 
\begin{equation}\begin{bmatrix} g^L(0) \\ g^L_t(0)\end{bmatrix} =  \begin{bmatrix} g_0 \\  g_1\end{bmatrix} + \int_0^\infty J(-s)\begin{bmatrix} 0 \\   (V-5\phi^4)g(s) \end{bmatrix} ds. \label{thetaLdata}\end{equation}
We wish to compare $g$ and $g^L$.
By the decay property of $V, \, 5\phi^4$ and  the Strichartz estimate (\ref{eq:strichartzwithpotential}), we know the integral term in (\ref{thetaLdata}) converges in $\HL$.

Then we have 
\[\begin{bmatrix} g(t) \\ g_t(t)\end{bmatrix} - \begin{bmatrix} g^L(t) \\ g^L_t(t)\end{bmatrix} = - \int_t^\infty J(t-s)\begin{bmatrix} 0 \\   (V-5\phi^4)g(s) \end{bmatrix} ds.\]
In particular
\begin{equation}g(t) = g^L(t) - \int_t^\infty \frac{\sin ((t-s)|\nabla|)}{|\nabla|}  ( (V-5\phi^4)g(s) ) ds. \label{TTL}\end{equation}
Since $g\in  \RSTEn (\R^3\times[0,\infty) )$, by continuity of the norm in the time variable, we have  $\|g\|_{  \RSTEn (\R^3\times[T_0,\infty) )}\rightarrow 0$ as $T_0\rightarrow \infty$. Together with the fact $V-5\phi^4 \in L^{\frac{3}{2},1}$ and from H\"{o}lder, we obtain 
\[ \|(V-5\phi^4)g\|_{L^{3/2,1}_xL^{2}_t  ( \R^3 \times [T_0,\infty))}\rightarrow  0 \hspace{1cm} \text{ as } T_0\rightarrow \infty.\]
From  (\ref{Interpolation-2}), we also have 
\begin{align*}
\|(V-5\phi^4)g\|_{L^{6/5,2}_xL^{\infty}_t( \R^3 \times [T_0,\infty))} \lesssim \|V-5\phi^4\|_{L^{\frac32, 2}_x} \|g\|_{L^{6,\infty}_xL^\infty_t[T_0,\infty)} \rightarrow 0,  \hspace{0.2cm} T_0\rightarrow +\infty.
\end{align*}
 Now we can apply the Strichartz estimate (\ref{eq:strichartzwithpotential})  to (\ref{TTL}) which implies
\begin{align*} & \|g\|_{\RSTE
\cap \RSTEn (\R^3\times[T_0,\infty) )}\\ 
& \lesssim  \|g^L\|_{\RSTE
\cap \RSTEn (\R^3\times[T_0,\infty) )} +  \|(V-5\phi^4)g\|_{\DRST ( \R^3 \times [T_0,\infty))}\\
&\longrightarrow  0 \hspace{1cm } \text{ as } T_0\rightarrow +\infty. 
\end{align*}

Hence we can pick $T_*$ large enough such that  $ \|g\|_{\RSTE
\cap \RSTEn (\R^3\times[T_0,\infty) )} <\epsilon$ for $T_0>T_*$.  Combining this with the difference estimate 
\[\|f(t,x) -g(t,x)\|_{\RSTE
\cap \RSTEn (\R^3\times[T_0,\infty) )}\lesssim \|(f_0,f_1) -(g_0, g_1)\|_{\HL}\lesssim \epsilon,\]
we get
\[\|f\|_{\RSTE
\cap \RSTEn (\R^3\times[T_0,\infty) )} <\epsilon, \hspace{1cm} \text{ for } T_0>T_*. \]
We have proved (\ref{largeT}).

In a similar fashion, we  consider  the  free wave $f^L(t,x)$  with   the initial data 
\begin{equation}\begin{bmatrix} f^L(0) \\ f^L_t(0)\end{bmatrix} =  \begin{bmatrix} f_0 \\  f_1\end{bmatrix} + \int_0^\infty J(-s)\begin{bmatrix} 0 \\   (V-5\phi^4)f(s) \end{bmatrix} \, ds. \end{equation}
We know  the integral term here converges in $\HL$ and 
\begin{equation}f(t) = f^L(t) - \int_t^\infty \frac{\sin ((t-s)|\nabla|)}{|\nabla|}  ( (V-5\phi^4)f(s) ) \, ds. \end{equation}
Now that we have already proved $\|f\|_{\RSTE
\cap \RSTEn (\R^3\times[T_0,\infty) )} \rightarrow 0 $ as $T_0\rightarrow +\infty,$  we can apply Strichartz to obtain 
\begin{align*}\|\OR{f}(t,x) -\OR{f}^L(t,x)\|_{\HL} &\lesssim \|(V-5\phi^4)f(s,x)\|_{L^{6/5,2}_xL^{\infty}_s\cap L^{3/2,1}_xL^{2}_s ( \R^3 \times [t,\infty))} \\
&\lesssim \|V-5\phi^4\|_{L^{\frac32,1}_x} \|f(s,x)\|_{L^{6,2}_xL^{\infty}_s\cap L^{\infty}_xL^{2}_s ( \R^3 \times [t,\infty))}\\
&\rightarrow  0  \hspace{1cm} \text{ as } t\rightarrow +\infty. 
 \end{align*}
 This establishes~(\ref{f-scatter}).
 
 \begin{remark} Due to the near optimal decay assumption on our potential~$V$, we can not apply the structure formula from~\cite{BS} to obtain scattering for solutions to the wave equation with potential.   The proof above seems to provide a new perspective:  scattering to a free wave occurs because the potential term becomes negligible for large times. This insight requires the use of reverse Strichartz estimates. 
   \end{remark}

\section{Channel of energy inequality}\label{sec:4}

In this section, we first prove the channel of energy estimate for solutions to the linear wave equation with potential if the initial data has a dominating  discrete mode. Then we show this estimate also holds for equation (\ref{eq:mainequation}) as long as the initial data is small enough.    Finally, for data which has a nontrivial but not dominant  discrete mode, we
 prove a growth lemma which ensures that  once we require the initial data to be sufficiently small, we can find a large time at which the solution is still small and the  discrete mode becomes dominant.

For the following basic perturbation result, we refer the reader to~\cite[Lemma 2.1]{JiaLiuXu} for   proof. 

\begin{lemma}\label{lem:perturbation}
Let $0\in I\subset \R$ be an interval of time. Suppose
$\tilde{u}(t,x)\in C_t(I,\dot{H}^1(\R^3))$ with
$\|\tilde{u}\|_{L^5_tL^{10}_x(I\times \R^3)}\leq M<\infty$,
$\|a\|_{L^{5/4}_tL^{5/2}_x(I\times \R^3)}\leq \beta<\infty$ and
$e(t,x),\,f(t,x)\in L^1_tL^2_x(I\times \R^3)$, satisfy
\begin{equation}
\partial_{tt}\tilde{u}-\Delta \tilde{u}+a(t,x)\tilde{u}+\tilde{u}^5=e,
\end{equation}
with initial data
$\overrightarrow{\tilde{u}}(0)=(\tilde{u}_0,\tilde{u}_1)\in
\dot{H}^1\times L^2$. Suppose for some sufficiently small positive
$\epsilon<\epsilon_0=\epsilon_0(M,\beta)$,
\begin{equation}
\||e|+|f|\|_{L^1_tL^2_x(I\times
\R^3)}+\|(u_0,u_1)-(\tilde{u}_0,\tilde{u}_1)\|_{\dot{H}^1\times
L^2}<\epsilon.
\end{equation}
Then there is a unique solution $u\in C(I,\dot{H}^1)$ with
$\|u\|_{L^5_tL^{10}_x(I\times \R^3)}<\infty$, satisfying the
equation
\begin{equation}
\partial_{tt}u-\Delta u+a(t,x)u+u^5=f,
\end{equation}
with initial data $\OR{u}(0) =(u_0,u_1)$.
Moreover, we have the following estimate
\begin{equation}
\sup_{t\in
I}\|\overrightarrow{u}(t)-\overrightarrow{\tilde{u}}(t)\|_{\dot{H}^1\times
L^2}+\|u-\tilde{u}\|_{L^5_tL^{10}_x(I\times
\R^3)}<C(M,\beta)\epsilon.
\end{equation}
\end{lemma}

We also need the   following result on the precise asymptotics of eigenfunctions corresponding to negative eigenvalues of the Schr\"{o}dinger operator $-\Delta-V$, which is a consequence of Theorem~4.2 in  Meshkov~\cite{Meshkov}. 

\begin{lemma} 
\label{Meshkov-lemma}
Let $V$ satisfy $\displaystyle{\sup_{x\in \R^3} (1 + |x|)^{\beta} |V (x)| <\infty }$ for some $\beta >2$, and suppose that $\rho\not\equiv0$
 is an eigenfunction corresponding to the eigenvalue $-k^2$   of   $-\Delta - V$. Then there exists $f \in L^2(\mathbb{S}^{2})$ which does not vanish identically, such that
  \begin{equation}
\rho(x) = e^{-k|x|}|x|^{-1}\left(f(\frac{x}{|x|})+\omega(x)\right),\label{eq:rho-expansion}
\end{equation}
where $\omega(x)$ satisfies
\begin{equation} \int_{\mathbb{S}^2}|\omega(R\theta)|^2\, d\sigma(\theta) = O(R^{-\frac12} ),  \hspace{1cm}  \text{as }  R \rightarrow +\infty. \label{eq:omega} \end{equation}
\end{lemma}

An important observation in~\cite{DKM3} is that the above precise asymptotics implies the following
channel of energy inequality for the associated linear wave equation.

\begin{lemma}\label{lem:linear-channel}
Let $V$ satisfy $\displaystyle{\sup_{x\in \R^3} (1 + |x|)^{\beta} |V (x)| <\infty }$ for some $\beta >2$, and suppose that $\rho\not\equiv0$
is an eigenfunction corresponding to the eigenvalue $-k^2$   of the operator $-\Delta - V$. Suppose that $u$ solves the equation
\[u_{tt}-\Delta u -Vu=0\]
with $\OR{u}(0) = \mu^+(\rho, k\rho)$, then for any $R>0$ the following channel of energy estimate holds for some constant $c(\rho, V,R)>0$
\begin{equation}
  \int_{|x|\geq t+R}  |\partial_t u|^2(x,t)\, dx \geq c(\rho, V,R) |\mu^+|^2, \hspace{1cm} \text{ for } t\geq 0. \label{eq:linear-channel}
\end{equation}
Similarly,  if $\OR{u}(0)=\mu^- (\rho, -k\rho)$,  then
\begin{equation}
 \int_{|x|\geq |t|+R}  |\partial_t u|^2(x,t)\, dx \geq c(\rho, V,R) |\mu^-|^2,  \hspace{1cm} \text{ for } t\leq 0. \label{eq:linear-channel-2}
\end{equation}
\end{lemma}

\begin{proof}  We first prove the lemma for initial data $\OR{u}(0) = \mu^+(\rho, k\rho)$. In this case  the solution $u$ has the explicit form
\[u(t, x) = \mu^+ e^{kt}\rho.\]
From  (\ref{eq:omega}), we can take $r_0$ large enough such that when $r>r_0$, we have  \begin{equation}
\int_{\mathbb{S}^2} \left|\omega(r\theta)\right|^2\, d\sigma(\theta)< \frac{1}{10}\int_{\mathbb{S}^2} \left|f(\theta)\right|^2 d\sigma(\theta). \label{eq:R0}\end{equation}
By the asymptotics of $\rho$ in (\ref{eq:rho-expansion}), we get that
\begin{eqnarray*}
\int_{|x|\geq t+R} |\partial_t u|^2(x,t)\, dx
&\ge& \int_{r\geq t+R+r_0}\int_{\mathbb{S}^{2}} |\mu^+ k|^2 e^{-2k(r-t)} \left(f(\theta) + \omega(r\theta)\right)^2\,d\sigma(\theta) dr\\
&\gtrsim&\int_{R+r_0}^{\infty}\int_{\mathbb{S}^2}|\mu^+k|^2e^{-2kr}|f(\theta)|^2\,d\sigma(\theta) dr.
 \end{eqnarray*}
Then (\ref{eq:linear-channel}) follows.

The case when  $\OR{u}(0) = \mu^-(\rho, -k\rho)$ is similar, and we omit the detail.
 \end{proof}

 Lemma~\ref{lem:linear-channel}  can be generalized to the case when the initial data has finitely many discrete modes.
 
\begin{lemma}
\label{lem:Lchannel}
Let $V$ satisfy $\displaystyle{\sup_{x\in \R^3} (1 + |x|)^{\beta} |V (x)| <\infty }$ for some $\beta >2$,  and suppose that
 $ -\Delta -V $  has negative    eigenvalues $-k_1^2 \leq -k_2^2 \leq \ldots \leq -k_n^2 <0$ with corresponding orthonormal eigenmodes $\rho_1, \rho_2, \ldots, \rho_n$.
Suppose that $u$ solves the equation
\EQ{\label{eq:lin u}
u_{tt}-\Delta u -Vu=0
}
with initial data  $\OR{u}(0) = \sum_{i=1}^n \mu^+_i(\rho_i,k_i\rho_i) $,  
then for any $R>0$, there exists a constant $c(R)>0$ such that we have the following channel of energy estimate forward in time
 \begin{equation}
 \int_{|x|\geq t+R}  |\partial_t u|^2(x,t)\, dx \geq  c(R) \sum_{i=1}^n |\mu_i^+|^2, \hspace{1cm}\text{ for } t >0.  \label{eq:Lchannel1}
 \end{equation}
Similarly, if we consider data of the form  $\OR{u}(0) = \sum_{i=1}^n \mu^-_i (\rho_i, -k_i\rho_i),$   the channel of energy estimate holds backward in time.
\end{lemma}
\begin{proof} 
It suffices to prove the lemma for sufficiently large $R>0$. By normalizing the coefficients, we will prove (\ref{eq:Lchannel1}) when $\sum_{i=1}^n |\mu^+_i|^2=1$. We divide the proof into several steps.

\noindent
\textit{Step 0: Computing the asymptotics.}

First notice that  the solution has an explicit formula
\[u=\sum_{i=1}^n \mu_i^+e^{k_it}\rho_i .\]
 From Lemma~\ref{Meshkov-lemma}, we know that each $\rho_i$ has the following asymptotic
\[
\rho_i = e^{-k_i|x|}\frac{1}{|x|}\left(f_i(\frac{x}{|x|})+\omega_i(x)\right)
\] 
with $f_i\in L^2(\mathbb{S}^2)$ which does not vanishing identically, and  $\omega_i$ satsifies (\ref{eq:omega}).

Now given any $R>0$, using Lemma~\ref{Meshkov-lemma} we have
\begin{eqnarray*}
&& \lim_{t\rightarrow +\infty}\int_{|x|\geq t+R}  |\partial_t u|^2(x,t)\, dx\\
&&\,\,=   \lim_{t\rightarrow +\infty} \int_{r>t+R}\int_{\theta\in\mathbb{S}^2} \left[\sum_{i=1}^n \mu_i^+k_i e^{-k_i(r-t)}\left(f_i(\theta)+\omega_i(r\theta)\right) \right]^2\,d\sigma(\theta) \,dr\\
&&\,\,= \lim_{t\rightarrow +\infty} \int_{r>R}\int_{\theta\in\mathbb{S}^2} \left[\sum_{i=1}^n \mu_i^+k_i e^{-k_ir}\left(f_i(\theta)+\omega_i((r+t)\theta)\right) \right]^2\,d\theta \,dr\\
&&\,\,=\int_{r>R}\int_{\theta\in\mathbb{S}^2} \left[\sum_{i=1}^n \mu_i^+k_i e^{-k_ir}f_i(\theta) \right]^2\,d\sigma(\theta) \,dr.
\end{eqnarray*}
Here we used the decay condition (\ref{eq:omega}) for $\omega_i$. 

\noindent
\textit{Step 1: lower bound for the asymptotics. }

We claim that  for any $R\ge0$ fixed, there exists constant $c(R)>0$ such that  for any $\mu^+_i$ satisfying $\sum_{i=1}^n |\mu^+_i|^2=1$, we have
\begin{equation}\int_{r>R}\int_{\theta\in\mathbb{S}^2} \left[\sum_{i=1}^n \mu_i^+k_i e^{-k_ir}f_i(\theta) \right]^2\,d\sigma(\theta) \,dr   \geq c(R).   \label{eq:Lchannel}\end{equation}
Suppose (\ref{eq:Lchannel}) is not true, then
for  any $ N>0$, we find  $\mu_i^+(N)$  satisfying $\sum_{i=1}^n |\mu^+_i(N)|^2=1$  such that
\begin{equation}\int_{r>R}\int_{\theta\in\mathbb{S}^2} \left[\sum_{i=1}^n \mu_i^+(N)k_i e^{-k_ir}f_i(\theta) \right]^2\,d\sigma(\theta) \,dr   < \frac{1}{N}. \label{eq:NN}\end{equation}

Using that  $\mu^+_i(N)$ are bounded,  we can extract a convergent subsequence. Hence we can assume that $\mu^+_i(N)\rightarrow a_i$ as $N\rightarrow \infty$, and $\sum_{i=1}^n a_i^2=1$.  By the dominated convergence theorem,  we pass to the limit in (\ref{eq:NN}) and  get
\[
\int_{r>R}\int_{\theta\in\mathbb{S}^2} \left[\sum_{i=1}^n a_ik_i e^{-k_ir}f_i(\theta) \right]^2\,d\sigma(\theta) \,dr=0
\] 
which  implies that
\begin{equation}\sum_{i=1}^n a_ik_i e^{-k_ir}f_i(\theta) =0 \hspace{1cm} \text{ for }\,  r>R, \,
\theta\in \mathbb{S}^2.\label{eq:ContradictionId}\end{equation}
Now we consider the problem in several cases:

\smallskip

\textbf{Case 1:} if $k_i$ are different, then in (\ref{eq:ContradictionId}) 
we first multiply with $e^{-k_nr}$ and let $r\rightarrow \infty$, we conclude $a_nf_n=0$, and similarly we conclude
\[
a_i\,f_i(\theta)=0,  \hspace{1cm}{\rm for}\,\,1\leq i\leq n\,\,{\rm and}\,\,\theta\in \mathbb{S}^2.
\]
Since $\|f_i\|_{L^2(\mathbb{S}^2)}\not=0,$ we conclude that $a_i=0, $ which is a contradiction to   $\sum a_i^2=1$.

\smallskip

\textbf{Case 2: }  If one of the eigenvalues has multiplicity more than 1, say, $k_{i_0}$ with multiplicity $m$, i.e., 
$k_{i_0}= k_{i_0+1}=k_{i_0+m-1}\not =k_j$ for any $j\in \{1, \ldots, n\}\backslash\{i_0, i_0+1, \ldots i_0+m-1\}$.  
All other eigenvalue still have multiplicity~$1$.  Then (\ref{eq:ContradictionId}) now reads as
 \[ a_1k_1 e^{-k_1r}f_1(\theta)   + \ldots  + 
 e^{-k_{i_0}r}k_{i_0}\left[\sum_{i=i_0}^{i_0+m-1} a_i  f_i(\theta)\right] + \ldots +a_nk_n e^{-k_nr}f_n(\theta) =0 
 \hspace{0.5cm} \text{ for } \,\,r >R, \,\,\theta\in \mathbb{S}^2.\]
 Applying the same method as in Case~1, we conclude that
 \[
 a_1 f_1(\theta)=0, \ldots,  \sum_{i=i_0}^{i_0+m-1} a_i   f_i(\theta)=0, \ldots ,a_n f_n(\theta) =0, \hspace{1cm}{\rm for}\,\, \theta\in \mathbb{S}^2, 
 \]
  which implies $a_i=0$,  for any $i\in \{1, \ldots, n\}\backslash\{i_0, i_0+1, \ldots i_0+m-1\}$.

Now we consider the part  $\sum_{i=i_0}^{i_0+m-1} a_i   f_i(\theta)=0$ and prove that all $a_i=0$.  
Denote $L=-\Delta-V$. By $L_{\phi}\rho_i =-k_{i_0}^2\rho_i,  i_0\leq i\leq i_0+m-1$,  we see that
 \[
 L\left(\sum_{i=i_0}^{i_0+m-1} a_i  \rho_i\right) = -k_{i_0}^2\left(\sum_{i=i_0}^{i_0+m-1} a_i \rho_i\right).
 \]
 Assuming towards a contradiction that not all $a_i=0$, we conclude that $\sum_{i=i_0}^{i_0+m-1} a_i  \rho_i $ is an  eigenfunction for $L$ with eigenvalue $-k_{i_0}^2$.

On the other hand,  
 \[\sum_{i=i_0}^{i_0+m-1} a_i  \rho_i  = e^{-k_{i_0}|x|}\frac{1}{|x|}\left [\sum_{i=i_0}^{i_0+m-1} a_i w_i(x)\right].\]
This  contradicts Lemma~\ref{Meshkov-lemma}, in particular (\ref{eq:rho-expansion}).  Hence we conclude that $a_i=0, 1\leq i\leq n$, which is a contradiction to $\sum a_i^2=1$.

\smallskip

\textbf{Case 3:} In general, we could have several eigenvalues that have multiplicity  more than~$1$. In that case we repeat  the argument in Case~2 as needed.

Hence we conclude that our claim (\ref{eq:Lchannel}) is true. 

\medskip 

\noindent

\textit{Step 2: Refining the lower bound for asymptotics. }

\smallskip

Next we refine (\ref{eq:Lchannel}) by obtaining  a better lower bound.  Let $\alpha_i = k_ie^{-k_ir} f_i(\theta), 1\leq i\leq n$ for $r>0,\,\theta\in \mathbb{S}^2$,  and
\[
\left< {\alpha}_i ,  {\alpha}_j\right>:=\int_{r>0}\int_{\theta\in \mathbb{S}^2}  {\alpha}_i   {\alpha}_j\,d\sigma(\theta) dr, \hspace{1cm} \mathcal{A}_{n\times n}:=[\left< {\alpha}_i ,  {\alpha}_j\right>]_{1\leq i, j\leq n}.
\]
Then (\ref{eq:Lchannel}) with $R=0$ implies that $\mathcal{A}$ is a positive definite matrix.  
And for any $\vec{v}\in \R^n, \|\vec{v}\|=1$, one has  $\vec{v}^t A \vec{v}\geq c(0) >0$.

Now for any $R>0$, we    change   variables $r=s+R$  in (\ref{eq:Lchannel}) to wit
 \begin{align} 
 \int_{r>R}\int_{\theta\in\mathbb{S}^2} \left[\sum_{i=1}^n \mu_i^+k_i e^{-k_ir}f_i(\theta) \right]^2\,d\sigma(\theta) \,dr = &\int_{s>0}\int_{\theta\in\mathbb{S}^2} \left[\sum_{i=1}^n \mu_i^+k_i e^{-k_i s} e^{-k_iR}f_i(\theta) \right]^2\,d\sigma(\theta) \,ds \notag \\
= & \sum_{i,j} \mu^+_i e^{-k_iR}  \mu^+_j e^{-k_jR} \left<\alpha_i,\alpha_j\right>\notag \\
\geq & c(0) \sum_{i=1}^n  \left|\mu^+_i e^{-k_iR}\right|^2. \label{eq:betterbound}
\end{align}

\smallskip

\noindent
\textit{Step 3: Channel of energy estimate}

\smallskip

Now we prove  (\ref{eq:Lchannel1}). The computation from {\it Step 0} implies that
\EQ{ 
& \int_{|x|\geq t+R}  |\partial_t u|^2(x,t)\, dx\\
&=   \int_{r>R}\int_{\theta\in\mathbb{S}^2} \left[\sum_{i=1}^n \mu_i^+k_i e^{-k_ir}\left(f_i(\theta)+\omega_i((r+t)\theta)\right) \right]^2\,d\sigma(\theta) \,dr\\
}
Expanding the square, this further equals 
\EQ{ 
&=\sum_{i,j=1}^n\int_{r>R}\int_{\theta\in\mathbb{S}^2} \mu_i^+\mu_j^+k_ik_j e^{-(k_i+k_j)r} f_i(\theta) f_j(\theta)\,d\sigma(\theta) \,dr \\
&\,\,\,+ \sum_{i,j=1}^n \int_{r>R}\int_{\theta\in\mathbb{S}^2}  \mu_i^+\mu_j^+k_ik_j e^{-(k_i+k_j)r}\bigg[ f_i(\theta)\omega_j((r+t)\theta) + f_j(\theta)\omega_i((r+t)\theta) \bigg]\,d\theta \,dr  \\
&\,\,\,+ \sum_{i,j=1}^n \int_{r>R}\int_{\theta\in\mathbb{S}^2}  \mu_i^+\mu_j^+k_ik_j e^{-(k_i+k_j)r}\ \omega_i((r+t)\theta) \omega_j((r+t)\theta) \,d\sigma(\theta) \,dr.
}
Using the decay estimate of $\omega_j$ in (\ref{eq:omega}) and Cauchy-Schwarz inequality, we infer that
\begin{align*}
&\sum_{i,j=1}^n\left|\int_{r>R}\int_{\theta\in\mathbb{S}^2}  \mu_i^+\mu_j^+k_ik_j e^{-(k_i+k_j)r} f_i(\theta)\omega_j((r+t)\theta) \,d\sigma(\theta) \,dr\right| \\
&\leq  \sum_{i,j=1}^n \left( \int_{r>R}\int_{\theta\in\mathbb{S}^2}  |\mu_i^+k_i|^2  e^{-2k_ir} \left|f_i(\theta)\right|^2 \,d\sigma(\theta) \,dr\right)^{\frac12} \\
&\hspace{2cm}\times \left( \int_{r>R}\int_{\theta\in\mathbb{S}^2}  |\mu_j^+k_j|^2  e^{-2k_jr} \left|\omega_j((r+t)\theta)\right|^2 \,d\sigma(\theta) \,dr\right)^{\frac12} \\
&\lesssim R^{-\frac{1}{4}}\sum_{i=1}^n\left|\mu_i^+\right|^2e^{-2k_iR}.
\end{align*}
Similarly,  we have
\begin{align*}  & \sum_{i,j=1}^n\int_{r>R}\int_{\theta\in\mathbb{S}^2}  \mu_i^+\mu_j^+k_ik_j e^{-(k_i+k_j)r}\ \omega_i((r+t)\theta) \omega_j((r+t)\theta) \,d\sigma(\theta) \,dr\\
 &\lesssim  R^{-\frac{1}{2}}\sum_{i=1}^n\left|\mu_i^+\right|^2e^{-2k_iR}.\end{align*}
Together with (\ref{eq:betterbound}) we obtain 
\begin{align*}
 \int_{|x|\geq t+R}  |\partial_t u|^2(x,t)\, dx
&\geq    \,c(0) \sum_{i=1}^n  \left|\mu^+_i\right|^2 e^{-2k_iR}-  C(R^{-\frac14} +R^{-\frac12})\sum_{i=1}^n  \left|\mu^+_i\right|^2 e^{-2k_iR}\\
&\geq\, \frac{c(0)}{2}\sum_{i=1}^n \left |\mu^+_i e^{-k_iR}\right|^2  \geq \, \frac{ c(0)}{2}e^{-2k_1R}, 
\end{align*}
where $R$ is sufficiently large. The lemma is proved.
 \end{proof}

Next we consider the case when there are several negative eigenvalues and prove that if one of the discrete modes is dominant, then we still have the channel of energy estimate.
 
 \begin{corollary}
 \label{Cor:linear-channel} 
 
Let $V$ satisfy $\displaystyle{\sup_{x\in \R^3} (1 + |x|)^{\beta} |V (x)| <\infty }$ for some $\beta >2$,  and suppose that
 $ -\Delta -V $ has no zero eigenvalue or zero resonance, and that it has negative    eigenvalues $-k_1^2 \leq -k_2^2 \leq \ldots \leq -k_n^2 <0$ with corresponding orthonormal eigenmodes $\rho_1, \rho_2, \ldots, \rho_n$.

Let $ u(t)$ be  a solution to \eqref{eq:lin u} with initial data \[\OR{u}(0) =  (\gamma_0, \gamma_1) + \sum_{i=1}^n [\mu^+_i (\rho_i, k_i\rho_i) +\mu^-_i(\rho_i, -k_i\rho_i)]\]
satisfying the  orthogonal conditions $\int \rho_i\gamma_0\, dx =\int \rho_i\gamma_1\, dx = 0$,   $1\leq i\leq n$.

(1) For any $R\ge 0$, if we  have
\begin{equation}
|\mu^+_{i_0}|> K_0\left[\|(\gamma_0, \gamma_1)\|_{\HL}  +\sum_{i=1}^n|\mu_i^-|\right]  \label{eq:KK}
\end{equation}
for sufficiently large constant $K_0:=K_0(R)>0$,  then there exists a constant $c(R)>0$ such that
\begin{equation}
\int_{|x|\geq t+R} |\partial_t u|^2(x,t)\, dx \geq\, c(R)\, |\mu^+_{i_0}|^2, \hspace{1cm} \text{ for\,\,all\,\, } t\geq 0. \end{equation}

(2) For any $R\ge 0$, if we have
\[
|\mu^-_{i_0}|> K_0\left[\|(\gamma_0, \gamma_1)\|_{\HL}  +\sum_{i=1}^n|\mu_i^+|\right]  \label{eq:K}\]
for sufficiently large fixed constant $K_0:=K_0(R)>0$, then  there exists a constant $c(R)>0$ such that
\begin{equation}
\int_{|x|\geq |t|+R} |\partial_t u|^2(x,t)\, dx \geq\, c(R) \, |\mu^-_{i_0}|^2, \hspace{1cm} \text{ for\,\,all\,\, } t\leq 0. \end{equation}
 \end{corollary}

\begin{proof}
To prove (1), first note that the solution is of the form
\[u= \sum_{i=1}^n \mu_i^+ e^{k_it}\rho_i +\mu_i^-e^{-k_it}\rho_i +\gamma(t,x)\]
with the continuous part $\gamma$ solving the equation
\[\gamma_{tt} +P^{\perp}(-\Delta -V)\gamma=0\]
Hence from Lemma~\ref{lem:Lchannel} and the Strichartz estimate for $\gamma$ (\ref{eq:strichartzwithpotential}), we get for $ t\geq 0$
\begin{align*}
&\int_{|x|>t+R} |\partial_t u(x,t)|^2 \,dx \\
&\geq  \, \frac{1}{2} \int_{|x|\geq t+R} \left|\sum_{i=1}^n\mu_{i}^+k_{i} e^{k_{i}t}\rho_i\right|^2dx
  -2 \sum_{i=1}^n\int_{|x|\geq t+R}\left|  \mu_i^- k_ie^{-k_it} \rho_i\right|^2 dx  -2 \int_{|x|\geq t+R} |\partial_t \gamma|^2 dx \\
&\geq \, c(R) \sum_{i=1}^n |\mu_{i}^+|^2 - C \sum_{i=1}^n   |\mu_i^-|^2  - C \|(\gamma_0,\gamma_1)\|^2_{\HL}\\
&\geq\,  \frac{c(R) }{2} |\mu_{i_0}^+|^2,
\end{align*}
if $K_0$ in (\ref{eq:KK}) is sufficiently large.\\
Case (2) follows from (1) by time reversal. 
\end{proof}

Next we shall see that the channel of energy estimate is stable with respect to nonlinear perturbations.  In particular, the following lemma shows that if the initial data is very close to a steady state, and
  one discrete eigenmode of the initial data is dominant,  then the solution will radiate energy outside the light cone either forward or backward in time.

\begin{lemma}\label{nonlinear} 
Fix any $R>0$. Consider a finite energy solution $u$ to the nonlinear  equation (\ref{eq:mainequation}) with initial data $(u_0, u_1)\in \HL$.  
Given a stationary solution $\phi$ and  $\LL_{\phi} = -\Delta -V +5\phi^4$ with orthonormal eigenmodes $\rho_1, \rho_2, \ldots, \rho_n$ corresponding to negative eigenvalues $-k_1^2 \leq -k_2^2 \leq \ldots \leq -k_n^2 <0.$

(1) Let  $(u_0, u_1)$ be of the form 
$\OR{u}(0) = (\phi, 0) + (h_0,h_1) $ with 
\[(h_0,h_1) =  (\gamma_0, \gamma_1) + \sum_{i=1}^n \left[\mu^+_i (\rho_i, k_i\rho_i) +\mu^-_i(\rho_i, -k_i\rho_i)\right]\]
and   $\int \rho_i \gamma_0\, dx=\int \rho_i \gamma_1\, dx =0$ for all $1\leq i\leq n$.
Assume that $$|\mu_{i_0}^+|:=\max\{|\mu_{i}^+|,\,i=1,\dots,n\}$$ and that
\begin{align}
|\mu^+_{i_0}|>  K\left[\|(\gamma_0, \gamma_1)\|_{\HL} +  \sum_{i=1 }^n |\mu_i^-|\right], \label{eq:mu-domination}
\end{align}
as well as 
$$\|(h_0,h_1)\|_{\HL}<\epsilon_{\ast},$$
for some sufficiently large constants $K\gg 1$ and sufficiently small $\epsilon_{\ast}>0$ that only depend on the potential $V$ and $R$.
Then  the solution satisfies the channel of energy estimate
\begin{equation}
\int_{|x|\geq t+R} |\partial_t u|^2(x,t) \,dx \geq c(R) |\mu^+_{i_0}|^2, \hspace{1cm} \text{ for } t\geq 0 \label{eq:nonlinear-channel}
\end{equation}
for some constant $c(R)>0$.

 (2) Assume that $(u_0, u_1)$ has the   decomposition $\OR{u}(0) = (\phi, 0) +(h_0,h_1)$ with
 \[ (h_0,h_1) = \sum_{i=1}^n\mu^+_{i} (\rho_{i}, k_{i}\rho_{i}) +(\mathcal{R}_0,\,\mathcal{R}_1).\]
Furthermore, suppose that for  $|\mu_{i_0}^+|:=\max\{|\mu_i^+|,\,i=1,\dots,n\}$, we have
 \begin{equation}|\mu^+_{i_0}| > K\|(\mathcal{R}_0,\,\mathcal{R}_1)\|_{\HL} \label{eq:large1}\end{equation}
 and $\|(h_0,h_1)\|_{\HL}<\epsilon_{\ast}$,
for sufficiently large $K\gg 1$ and sufficiently small $\epsilon_{\ast}>0$ that depend only on $V,\,R$.  Then
 the solution $u$ satisfies the channel of energy estimate
 \begin{equation}
\int_{|x|\geq t+R} |\partial_t u|^2(x,t)\, dx \geq c(R) |\mu^+_{i_0}|^2, \hspace{1cm} \text{ for } t\geq 0 \label{eq:nonlinear-channel2}
\end{equation}
for some  constant $c(R)>0$.

(3) Similar results hold when we switch  $\mu_i^-$ with $\mu_i^+$ in (1), (2) and  consider $t\leq 0$. \end{lemma}
   
   \begin{proof}
   (1)  Write $u=\phi +h$.  Then $h$  solves the equation
   \[h_{tt} +(-\Delta -V +5\phi^4) h =\mathcal{N}(h,\phi)\]
   with $\mathcal{N}(h,\phi)=-(\phi+h)^5 + \phi^5 +5\phi^4h$.

Let    $h^L$ be the solution to the linear equation
      \begin{equation}h^L_{tt} +(-\Delta -V +5\phi^4) h^L =0. \label{eq:linearized}\end{equation}
Define
    \begin{equation}
    \widetilde{V}(x,t):= \left\{\begin{aligned}
 V(x) & & \text{ if  } |x|\geq |t|,\\
 0 & &  \text{ if  } |x| < |t|, 
 \end{aligned}\right.
  \end{equation}
  and 
  \begin{equation}
  \widetilde{\phi}(x,t):= \left\{\begin{aligned}
 \phi(x) & & \text{ if  } |x|\geq |t|,\\
 0 & &  \text{ if  } |x| < |t|, 
 \end{aligned}\right.\end{equation}
 respectively. 
Let $\tilde{h}^L$ and $\tilde{h}$ be the solution to the linear and nonlinear wave equation with truncated potential, viz.
\begin{align}\tilde{h}^L_{tt} +(-\Delta -\widetilde{V} +5\widetilde{\phi}^4) \tilde{h}^L & =0,\\
 \tilde{h}_{tt} +(-\Delta -\widetilde{V} +5\widetilde{\phi}^4) \tilde{h} &=\mathcal{N}(\tilde{h},\widetilde{\phi}).\label{eq:truncation}\end{align}
It is easy to check that
$$\widetilde{V},\,\widetilde{\phi}\in L^{\frac{5}{4}}_tL^{\frac{5}{2}}_x(\mathbb{R}^3\times \mathbb{R}).$$
We take the initial data
\[\OR{h}(0)=\OR{h}^L(0)=\OR{\tilde{h}}^L(0) =\OR{\tilde{h}}(0)= (\gamma_0, \gamma_1) + \sum_{i=1}^n [\mu^+_i (\rho_i, k_i\rho_i) +\mu^-_i(\rho_i, -k_i\rho_i)]\]
which satisfy the condition  (\ref{eq:mu-domination}) with a large constant $K$ to be chosen later. 
By finite speed of propagation,  $t\in \R$,  $\OR{h}=\OR{\tilde{h}}$, $\OR{h}^L =\OR{\tilde{h}}^L$ for $|x|>|t|$.
In view of  Lemma \ref{lem:perturbation},
  \begin{equation}\sup_{t\in [0,\infty)}\left\|\OR{\tilde{h}}(t)\right\|_{\HL} + \left\|\tilde{h}\right\|_{L_t^5L_x^{10}([0,\infty)\times \R^3)}\lesssim \|\OR{\tilde{h}}(0)\|_{\HL}\lesssim |\mu^+_{i_0}| \label{eq:difference}\end{equation}
  and
  \[\sup_{t\in [0,\infty)}\left\|\OR{\tilde{h}}(t)- \OR{\tilde{h}}^L(t)\right\|_{\HL} + \left\|\tilde{h} -\tilde{h}^L\right\|_{L_t^5L_x^{10}([0,\infty)\times \R^3)}\lesssim |\mu^+_{i_0}|^2,\]
if $\epsilon_{\ast}$ is chosen sufficiently small depending on $V$.

 Take $K>K_0(R)$ where $K_0(R)$ is the constant from part (1) of Corollary~\ref{Cor:linear-channel}, then we get that the linear solution $h^L$ satisfies the channel estimate,
  \[\int_{|x|\geq t+R} \left|\partial_t h^L(x,t)\right|^2 \,dx \geq c(R) |\mu_{i_0}^+|^2, \hspace{1cm}  \text{  for  }t\geq 0.\]
   Hence,  for all $t\geq 0$
  \begin{align*}\int_{|x|\geq t+R} |\partial_t h|^2(x,t)\,dx =& \int_{|x|\geq t+R} |\partial_t \tilde{h}|^2(x,t)\,dx \\\geq& \int_{|x|\geq t+R} |\partial_t \tilde{h}^L|^2(x,t)\,dx- C |\mu^+_{i_0}|^4\\
=& \int_{|x|\geq t+R} \left|\partial_t  {h}^L\right|^2(x,t)\,dx- C\, |\mu^+_{i_0}|^4\\
  \geq &  c(R)|\mu^+_{i_0}|^2 -C|\mu^+_{i_0}|^4 \geq  \frac{c(R)}{2}|\mu^+_{i_0}|^2.
  \end{align*}
  The last line holds provided  $ \epsilon_{\ast}=\epsilon_{\ast}(R)\gtrsim \left|\mu_{i_0}^+\right|$ is   small enough.

\medskip

\noindent
(2) Consider two solutions to equation (\ref{eq:mainequation}) $u$ and $v$, with data
\begin{eqnarray*}
 \OR{u}(0) &=& (\phi, 0) + \sum_{i=1}^n \mu^+_{i} (\rho_{i}, k_{i}\rho_{i}) + (\mathcal{R}_0,\,\mathcal{R}_1),\\
 \OR{v}(0) &=& (\phi, 0) +   \sum_{i=1}^n \mu^+_{i} (\rho_{i}, k_{i}\rho_{i}),
\end{eqnarray*}
respectively. 
  If we set $u=\phi +h$ and $v =\phi +\ell$, then $h, \ell $ satisfy
    \begin{align*} h_{tt} +(-\Delta -V +5\phi^4) h =& \mathcal{N}(h,\phi)\\
    \ell_{tt} +(-\Delta -V +5\phi^4) \ell = &\mathcal{N}(\ell,\phi)
    \end{align*}
with initial data
\[\OR{h}(0)=  \sum_{i=1}^n \mu^+_{i} (\rho_{i}, k_{i}\rho_{i})+(\mathcal{R}_0,\,\mathcal{R}_1), \hspace{1cm}\OR{\ell}(0) = \sum_{i=1}^n \mu^+_{i} (\rho_{i}, k_{i}\rho_{i}).\]
As  in the proof for  (1), we define $\widetilde{V}, \widetilde{\phi}$ and consider truncated versions $\tilde{h}, \tilde{\ell}$ 
that satisfy the equation (\ref{eq:truncation}), with data  $\OR{\tilde{h}}(0) =\OR{h}(0)$, $\OR{\tilde{\ell}}(0) =\OR{\ell}(0)$.  
Then from finite speed of propagation we infer  $\OR{h}=\OR{\tilde{h}}, \OR{\ell}=\OR{\tilde{\ell}}$ for $|x|\geq |t|. $
The perturbation lemma~\ref{lem:perturbation} and \eqref{eq:large1} yield the bound
\[\sup_{t\in [0,\infty)}\left\|\OR{\tilde{h}}(t)-\OR{\tilde{\ell}}(t)\right\|_{\HL}\leq C\|\OR{\tilde{h}}(0)-\OR{\tilde{\ell}}(0)\|_{\HL} \leq \frac{C}{K}|\mu^+_{i_0}|.\]
Note that
$$\|\OR{\ell}(0)\|_{\HL}\lesssim |\mu_{i_0}^+|.$$
From  part (1) we know that  there exists $\epsilon_{\ast}(R)>0$   small enough,  such that if $\|\OR{\ell}(0)\|<\epsilon_{\ast}$, then  $\ell(t,x)$ satisfy the channel of energy inequality
\[\int_{|x|\geq t+R} |\partial_t \ell |^2(x,t)\,  dx  \geq c(R) |\mu^+_{i_0}|^2 \hspace{1cm}\text{  for } t\geq 0.\]
Hence we get  for $t\geq 0$,
 \begin{align*}
\int_{|x|\geq t+R} |\partial_t h|^2(x,t)\,dx =&\int_{|x|\geq t+R} |\partial_t\tilde{h}|^2(x,t)\, dx\\
 \geq&   \int_{|x|>t+R} |\partial_t\tilde{\ell}|^2(x,t)\, dx - \frac{C^2}{K^2}|\mu_{i_0}^+|^2
\\
= &\int_{|x|\geq t+R}{|\partial_t \ell |^2}(x,t)\,dx   - \frac{C^2}{K^2}|\mu_{i_0}^+|^2\\
\geq & c(R) |\mu_{i_0}^+|^2 - \frac{C^2}{K^2}|\mu_{i_0}^+|^2
\geq    \frac{c(R)}{2} |\mu_{i_0}^+|^2.
\end{align*}
The last line holds if we pick $K:=K(R)$ large enough.

\noindent 

\smallskip 

 (3) The proof  is similar to  (1) and (2) and we omit the details here.
 \end{proof}

Initially,  the discrete spectral component may not be large enough as required by~(\ref{eq:large1}).
But since any  eigenmode grows exponentially either forward or backward in time, we might expect that it will take over the dispersive  term for large times as long as it is not too small initially. 
The following lemma makes this logic precise.

\begin{lemma}\label{lem:growth}  
Given a steady state solution $\phi$  to the nonlinear equation (\ref{eq:mainequation}),  suppose that $\LL_{\phi} = -\Delta -V +5\phi^4$ has orthonormal eigenmodes $\rho_1, \rho_2, \ldots, \rho_n$ corresponding to eigenvalues $-k_1^2 \leq -k_2^2 \leq \ldots \leq -k_n^2 <0$.
 Suppose that $\OR{u}$ is a solution to equation (\ref{eq:mainequation}) with initial data
\[\OR{u}(0) =(\phi, 0) +   \sum_{i=1}^n \left[\mu^+_i (\rho_i, k_i\rho_i) + \mu^-_i (\rho_i, - k_i\rho_i)\right]  + (\gamma_0, \gamma_1)\]
obeying the orthogonality  conditions $\int \rho_i \gamma_0\, dx =\int\rho_i\gamma_1\, dx =0$, for all $1\leq i \leq n$. Write the solution as
\[\OR{u}(t) =(\phi, 0) + \OR{h}(t).\]
 (1) Suppose    $$|\mu_{i_0}^+|:=\max\{|\mu_i^+|,\,i=1,\dots,n\} \geq  \kappa\|\OR{h}(0)\|_{\HL}$$ for some constant $\kappa>0$. Then for any $\epsilon_{\ast}>0, K>1,$ there exist   $\varepsilon(\kappa, \epsilon_{\ast}, K)>0$ sufficiently small and $T(\kappa, K)>0$ sufficiently large,  such that  if $\|\OR{h}(0)\|_{\HL}<\varepsilon$ then
\[\OR{h}(T) =
\sum_{i=1}^n e^{k_{i}T}\mu_{i}^+(\rho_{i}, k_{i}\rho_{i}) + (\mathcal{R}_0,\,\mathcal{R}_1),\] with
$$\|\OR{h}(T) \|_{\HL}<\epsilon_{\ast}$$
and
\[\|(\mathcal{R}_0,\mathcal{R}_1)\|_{\HL} \leq \frac{1}{K}e^{k_{i_0}T} |\mu_{i_0}^+|\|(\rho_{i_0}, k_{i_0}\rho_{i_0})\|_{\HL}.\]

(2)   Suppose    $$|\mu_{i_0}^-|:=\max\{\mu_i^-:\,i=1,\dots,n\} \geq  \kappa\|\OR{h}(0)\|_{\HL}$$ for some constant $\kappa>0$. Then for any $\epsilon_{\ast}>0, K>1,$ there exist   $\varepsilon(\kappa, \epsilon_{\ast}, K)>0$ sufficiently small and $T(\kappa, K)>0$ sufficiently large,  such that  if $\|\OR{h}(0)\|_{\HL}<\varepsilon$ then
\[\OR{h}(-T) =
\sum_{i=1}^n e^{k_{i}T}\mu_{i}^-(\rho_{i}, -k_{i}\rho_{i}) + (\mathcal{R}_0,\,\mathcal{R}_1),\] with $$\|\OR{h}(-T) \|_{\HL}<\epsilon_{\ast}$$
 and
\[\|(\mathcal{R}_0,\,\mathcal{R}_1)\|_{\HL} \leq \frac{1}{K}e^{k_{i_0}T} |\mu_{i_0}^-|\|(\rho_{i_0}, -k_{i_0}\rho_{i_0})\|_{\HL}.\]
\end{lemma}

\begin{proof}    The proof of (2) is again the time reversal of (1), so it suffices to consider the latter.

\noindent
\textit{Step 1: bound on $h$. }

Writing $u=\phi+h$, we see that  $h$ solves the equation (with $\mathcal{N}$ as above)
    \[h_{tt} +(-\Delta -V +5\phi^4) h =\mathcal{N}(h,\phi).\]
Let  $h^L$  be the solution to the linear equation
\begin{equation}\label{eq:thelinear11}
 h^L_{tt} +(-\Delta -V +5\phi^4) h^L =0
\end{equation}
with data  $\OR{h}^L(0)=\OR{h}(0)$.
We denote by $S(t)g$  the solution to the linear equation (\ref{eq:thelinear11}) with data $(0,g)$ for any $g\in L^2$.   By decomposing the data into continuous and discrete modes,  the Strichartz estimates (\ref{eq:strichartzwithpotential}) for the continuous modes, and the explicit formula for the evolution of discrete modes, we can find absolute constants $C, A\geq 1$ such that
\begin{align}
\sup_{\tau\in [0,t)}\|\OR{S}(\tau)g\|_{\HL} +\|S(\tau)g\|_{L^5_tL^{10}_x([0,t)\times \R^3)} \leq  & C e^{k_1t}\|g\|_{L^2} \label{eq:g} \\
\sup_{\tau\in [0,t)}\|\OR{h}^L(\tau)\|_{\HL} +\|h^L(\tau)\|_{L^5_tL^{10}_x([0,t)\times \R^3)} \leq & \frac{A}{8} e^{k_1t}\|\OR{h^L}(0)\|_{\HL}\label{eq:h-linear}
\end{align}
Denote $\epsilon:=\|\OR{h}(0)\|_{\HL}<\varepsilon$. Now on an interval $[0,T)$ with
$e^{3k_1T}\varepsilon$ sufficiently small, we will use a continuity argument to show that for $t\in [0,T)$
\begin{equation}
\sup_{\tau\in [0,t)}\|\OR{h}(\tau)\|_{\HL} +\|h(\tau)\|_{L^5_tL^{10}_x([0,t)\times \R^3)} \leq  A e^{k_1t}\|\OR{h}(0)\|_{\HL}. 
\label{eq:h-nonlinear1}
\end{equation}
In fact,  assuming that the bound (\ref{eq:h-nonlinear1}) holds for $0\leq t\leq t_0$ with some $0< t_0<T,$ we will show that we actually have
\begin{equation}
\sup_{\tau\in [0,t)}\|\OR{h}(\tau)\|_{\HL} +\|h(\tau)\|_{L^5_tL^{10}_x([0,t)\times \R^3)} \leq  \frac{A}{2} e^{k_1t}\|\OR{h}(0)\|_{\HL}, \hspace{0.5cm} \text{ for all } 0\leq t\leq t_0.
\label{eq:h-nonlinear}
\end{equation}
Then a simple continuity argument finishes the proof of proof of (\ref{eq:h-nonlinear1}).
From Duhamel's formula
\[h(\tau) =h^L(\tau) +\int_0^\tau S(\tau-s)\mathcal{N}(h,\phi)(s)\, ds\]
Denote  $F(\tau,x)= \int_0^\tau S(\tau-s)\mathcal{N}(h,\phi)(s)\, ds$, then from (\ref{eq:g}) we get
\EQ{ 
\sup_{\tau\in [0,t)}(\|F(\tau)\|_{\dot{H}^1} + \|\partial_\tau F(\tau)\|_{L^2}) \leq &  \sup_{\tau\in [0,t)} C \int_0^\tau e^{k_1(\tau-s)}\|\mathcal{N}(h,\phi)(s)\|_{L^2} \, ds\\
\leq & \sup_{\tau\in [0,t)}  C
 e^{k_1\tau}\|\mathcal{N}(h,\phi)(s)\|_{L^1_{s}L^2_x([0,\tau)\times \R^3)}  \\ \leq & C
 e^{k_1t}\|\mathcal{N}(h,\phi)(s)\|_{L^1_{s}L^2_x([0,t)\times \R^3)}
}
and
\EQ{
\|F\|_{L^5_\tau L^{10}_x([0,t)\times \R^3)} \leq & \left \|\int_0^t \chi_{\{\tau-s \geq 0\}}\|S(\tau-s)\mathcal{N}(h,\phi)(s)\|_{L^{10}_x}ds\right\|_{L^5_\tau[0,t) }\\
\leq & \int_0^t\left \|S(\tau-s)\mathcal{N}(h,\phi)(s)\right\|_{L^5_{\tau}L^{10}_x([s,t)\times\R^3)}\, ds \\
\leq & C \int_0^t e^{k_1(t-s)}\|\mathcal{N}(h,\phi)(s)\|_{L^2}\,   ds  \leq  C
 e^{k_1t}\|\mathcal{N}(h,\phi)(s)\|_{L^1_{s}L^2_x([0,t)\times \R^3)}.
}
Note that $|\mathcal{N}(h,\phi)(s)| \lesssim\sum_{j=2}^5 |\phi|^{5-j} |h|^j$.  Assuming the bound (\ref{eq:h-nonlinear1}) on $[0,t_0)$,  for any $t\in [0,t_0)$ we pick an integer $J_0\geq 0$ such that $ J_0 <t\leq J_0+1$.  This leads to 
\begin{align*}\|\phi^3h^2\|_{L^1_{s}L^2_x([0,t)\times \R^3)}= &\sum_{q=0}^{J_0}\|\phi^3h^2\|_{L^1_{s}L^2_x([q,q+1)\times \R^3)} +\|\phi^3h^2\|_{L^1_{s}L^2_x([J_0,t)\times \R^3)}  \\
\lesssim &\sum_{q=0}^{J_0} \|h\|^2_{L^5_{s}L^{10}_x([q,q+1)\times \R^3)} + \|h\|^2_{L^5_{s}L^{10}_x([J_0,t)\times \R^3)}\\
\lesssim & \sum_{q=0}^{J_0}(Ae^{k_1(q+1)}\epsilon)^2 +(Ae^{k_1t}\, \epsilon)^2
\lesssim  A^2\epsilon^2 e^{2k_1t}
 \end{align*}
 We can control the other terms in $\mathcal{N}(h,\phi)$ in an analogous fashion, whence 
 \[\|\mathcal{N}(h,\phi)(s)\|_{L^1_{s}L^2_x([0,t)\times \R^3)}\lesssim \sum_{j=2}^5(A\, \epsilon \, e^{k_1t})^j\hspace{1cm} \text{  for } 0\leq t<t_0.\]
Using (\ref{eq:h-linear}),  we therefore obtain 
 \begin{align*}
&\sup_{\tau\in [0,t)}\|\vec{h}(\tau)\|_{\HL} +\|h\|_{L^5_tL^{10}_x([0,t)\times \R^3)}\\
&\leq  \sup_{\tau\in [0,t)}\left\|\OR{h}^L(\tau)\right\|_{\HL} +\|h^L\|_{L^5_tL^{10}_x([0,t)\times \R^3)} + Ce^{k_1t} \|\mathcal{N}(h,\phi)(s)\|_{L^1_s L^2_x([0,t)\times\R^3)}\\
&\leq  \frac{A}{8}e^{k_1t}\epsilon + Ce^{k_1t}\left[\sum_{j=2}^5(A\, e^{k_1t}\, \epsilon)^j\right]\leq \frac{A}{2} e^{k_1t}\epsilon,
\end{align*}
provided $e^{3k_1T}\varepsilon\ll 1$ is sufficiently small. Hence (\ref{eq:h-nonlinear1}) holds on $[0,T)$ as long as  $T, \varepsilon$ satisfy the relation $e^{3k_1T}\varepsilon< \epsilon_1$ with a small fixed constant~$\epsilon_1$.

\noindent
\textit{Step 2: Decide the constants.}\\
Now we consider the linear solution $h^L$ with data
\[\OR{h}^L(0)=\sum_{i=1}^n \big[\,\mu^+_i (\rho_i, k_i\rho_i) + \mu^-_i (\rho_i, - k_i\rho_i)\,\big]  + (\gamma_0, \gamma_1),\]
 then we have the explicit formula  for the linear solution
\[\OR{h}^L(t) = \sum_{i=1}^n \mu^+_i e^{k_it} (\rho_i, k_i\rho_i) + \sum_{i=1}^n \mu^-_i e^{-k_it}(\rho_i, -k_i\rho_i) + \OR{\gamma}(t).\]
For any given $\kappa,\,K$, we can choose a large constant $T(\kappa, K)$ such that
\begin{align}\left\|\sum_{i=1}^n \mu^-_i e^{-k_iT}(\rho_i, -k_i\rho_i) + \OR{\gamma}(T)\right\|_{\HL}  \lesssim     \frac{T}{\kappa}|\mu^+_{i_0}|
 \leq \frac{1}{2K} |\mu^+_{i_0}| e^{k_{i_0}T} \|(\rho_{i_0}, k_{i_0}\rho_{i_0})\|_{\HL}\label{eq:T}
 \end{align}
Next from Duhamel's formula and the estimate of $\mathcal{N}$ in step 1, we have 
\begin{align}
\left\|\OR{h}(T)-\OR{h}^L(T)\right\|_{\HL} = & \left\| \int_0^T S(T-s)\mathcal{N}(h,\phi)(s)\, ds \right\|_{\HL} \notag\\
\leq & Ce^{k_1T}\left[\sum_{j=2}^5(A\epsilon e^{k_1T})^j\right] \notag
 <\frac{1}{2K}|\mu^+_{i_0}| \label{eq:smallN},
\end{align}
if $ e^{3k_1T}\varepsilon $ is sufficiently small.
 
Hence we have
$\OR{h}(T) =\sum_{i=1}^n \mu^+_i e^{k_iT} (\rho_i, k_i\rho_i)  +(\mathcal{R}_0,\,\mathcal{R}_1)$
with \[(\mathcal{R}_0,\,\mathcal{R}_1)=   \sum_{i=1}^n \mu^-_i e^{-k_it}(\rho_i, -k_i\rho_i) + \OR{\gamma}(t) + \OR{h}(T)-\OR{h}^L(T)\] and $$\|(\mathcal{R}_0,\,\mathcal{R}_1)\|_{\HL}\leq \frac{1}{K}  |\mu^+_{i_0}| e^{k_{i_0}T} \|(\rho_{i_0}, k_{i_0}\rho_{i_0})\|_{\HL}.$$
We also have
\[\|\OR{h}(T)\|_{\HL}\leq \sum_{i=1}^n e^{k_iT}|\mu_i^+| +\|(\mathcal{R}_0,\,\mathcal{R}_1)\|_{\HL} \lesssim e^{k_1T}\varepsilon<\epsilon_{\ast},\]
by choosing $\varepsilon$ sufficiently small.
\end{proof}

\begin{remark}\label{remark:largemode}
While part (1) of Lemma~\ref{lem:growth} guarantees that at time $T$   the unstable mode $$e^{k_{i_0}T}\mu_{i_0}^+(\rho_{i_0},k_{i_0}\rho_{i_0})$$ dominates the continuous part and the stable mode, we cannot be sure of its size compared to the other unstable modes, which might grow faster. However, we can easily conclude that the largest mode at time $T$,  say $e^{k_jT}\mu_j^+(\rho_{j},k_{j}\rho_{j})$,  satisfies    \[\left\|e^{k_jT}\mu_j^+(\rho_{j},k_{j}\rho_{j})\right\|_{\HL}\geq \frac{1}{n+1}\|\OR{h}(T)\|_{\HL}.\]
 \end{remark}

\section{Global center stable manifold of unstable excited states}
\label{sec:5}

In this section we prove our main result. Before giving the detailed proof, let us briefly summarize the main ideas in physical terms. The crucial fact that we establish can be explained roughly as follows. Take any solution $U(t)$ which scatters to an unstable steady state $\phi$. We have shown in Section~2 that in a small neighborhood of $\OR{U}(0)$ in the energy space $\HL$, there exists a local, finite co-dimensional manifold $\mathcal{M}$ such that if $\OR{u}(t)$ starts on the manifold, i.e., if  $\OR{u}(0)\in\mathcal{M}$, then $\OR{u}(t)$ stays close to $\OR{U}(t)$ for all positive times and scatters to $(\phi,0)$. On the other hand, if $\OR{u}(t)$ starts in a small neighborhood of $\OR{U}(0)$ but {\it off} the manifold, then
$$\sup_{t\ge 0}\left\|\OR{u}(t)-\OR{U}(t)\right\|_{\HL}\ge \epsilon_1>0,$$
no matter how small $\left\|\OR{u}(0)-\OR{U}(0)\right\|_{\HL}$ is.  Suppose that $\left\|\OR{u}(0)-\OR{U}(0)\right\|_{\HL}$ is sufficiently small, then dynamically $\OR{u}(t)$ will stay close to $\OR{U}(t)$ for a long time, say for $0\leq t\leq T_0$. Since $\OR{U}(t)$ scatters to $(\phi,0)$, we can write (in the energy space)
$$\OR{U}(t)\approx (\phi,\,0)+\OR{U}^L(t)$$
for large times. Hence for large $ t\leq T_0$,
$$\OR{u}(t)\approx (\phi,\,0)+\OR{U}^L(t)$$
in the energy space. After time $T_0$, $\OR{u}(t)$ starts to deviate from $\OR{U}(t)$ as $\OR{u}(0)\not\in\mathcal{M}$. By an expansion of the energy functional near the steady state, we shall show that the deviation is due to growth in the unstable mode. Then it is not hard to conclude that at a large time $T_1>T_0$, $\OR{u}(t)-(\phi,0)-\OR{U}^L(t)$ concentrates most of its energy in the discrete mode and has energy $\gtrsim \epsilon_1$. These arguments  finally set the stage for us to apply the channel of energy inequalities proved in the previous section. We will show that besides the radiated energy that $\OR{U}^L$ carries to spatial infinity, $\OR{u}(t)$ emits a {\it second radiation}. The total radiated energy for $\OR{u}(t)$ will therefore exceed the radiated energy for $\OR{U}(t)$ by a fixed amount. Now note that $\OR{u}(t)$ has almost the same amount of energy as $\OR{U}(t)$, a comparison argument of the energy in the local region then implies that $\OR{u}(t)$, having strictly less energy than $(\phi,0)$ in the local region, can no longer scatter to $(\phi,0)$. Hence, locally the set $\mathcal{M}_{\phi}$ of all initial data for which the solution scatters to $(\phi,0)$ coincide with $\mathcal{M}$. Thus the set $\mathcal{M}_{\phi}$ has a manifold structure. This is the key property showing that scattering to unstable steady states is non-generic.  

Now we turn to the main argument.
Let us first compute the expansion of energy around any steady state $(\phi,0)$.

\begin{lemma} \label{lem:expansion1}
Let  $(u_0,u_1)=  (\phi,0)+(\Lambda_0,\Lambda_1)$, where $(\Lambda_0,\Lambda_1)\in\HL$.  Assume that
$$\|\Lambda_0\|_{L^6(\R^3)} < \beta\ll 1,$$
then we have
\begin{equation}\mathcal{E}((u_0,u_1)) = \mathcal{E}(\phi, 0) +\frac12(\LL_\phi\Lambda_0, \Lambda_0) +\frac12 (\Lambda_1, \Lambda_1) + O(\beta^3), \label{eq:Eu0}
\end{equation}
where $\LL_{\phi} = -\Delta -V +5\phi^4$.\\
Suppose $\LL_{\phi}$ has orthonormal eigenmodes $\rho_1, \rho_2, \ldots, \rho_n$ corresponding to eigenvalues $-k_1^2 \leq -k_2^2 \leq \ldots \leq -k_n^2 <0$.
If we further decompose 
\begin{align}
(\Lambda_0,\Lambda_1)= &(X_0,X_1)+ (w_0,w_1),\label{eq:XX}\\
(w_0,w_1)= &\sum_{i=1}^n \left[\mu^{+}_i (\rho_i, k_i\rho_i) +\mu_i^{-} (\rho_i, -k_i\rho_i)\right] + (\gamma_0, \gamma_1),
\end{align}
 with $(X_0,\,X_1)\in\HL$ and the orthogonality condition $\int   \rho_j\gamma_0\, dx= \int   \rho_j\gamma_1\, dx =0$, for all $1\leq j\leq n.$ Then we have
 \begin{align}
 \mathcal{E}((u_0,u_1)) =&  \mathcal{E}(\phi, 0)  + \frac12\big[ (\LL_{\phi}X_0, X_0) +(X_1,X_1) \big]+\frac12\big[ (\LL_\phi \gamma_0,  \gamma_0)+ (\gamma_1,\gamma_1)\big]  - \sum_{i=1}^n 2\mu^+_i\mu^-_ik_i^2\notag\\ & +(\LL_{\phi} X_0 , w_0) 
 + (X_1, w_1)  + O(\beta^3).
 \label{eq:Eu1}
\end{align}
\end{lemma}
\begin{proof} The proof is by direct computation
\begin{align*}
\mathcal{E}((u_0,u_1)) =& \mathcal{E}((\phi, 0)+(\Lambda_0,\Lambda_1)) \\
=&  \int \frac{|\nabla \phi +\nabla \Lambda_0 |^2}{2} +\frac{| \Lambda_1|^2}{2} -\frac{V(\phi+\Lambda_0)^2}{2} +\frac{(\phi+\Lambda_0)^6}{6}\,dx\\
=& \int \frac12 \Big(  |\nabla \phi|^2   - \frac12 V\phi^2  +\frac16 \phi^6 \Big) \,dx +\int  (-\Delta \phi  - V\phi + \phi^5)\Lambda_0  \, dx\\
&+ \int \frac12  \Big[ |\nabla \Lambda_0|^2  - \frac12 V\Lambda_0^2  +\frac52 \phi^4\Lambda_0^2 + \frac12  |\Lambda_1 |^2 +\frac16  \sum_{j\geq 3}C_6^j \phi^{6-j}\Lambda_0^j \Big] \, dx\\
  =& \mathcal{E}(\phi, 0) +\frac12(\LL_\phi\Lambda_0, \Lambda_0) +\frac12 (\Lambda_1, \Lambda_1) + O(\beta^3).
\end{align*}
This finishes the proof of (\ref{eq:Eu0}).

Next we further  expand the energy functional using (\ref{eq:XX})
 \begin{align*}
&\mathcal{E}((u_0,u_1)) \\
&=  \mathcal{E}(\phi, 0) +\frac12\left(\LL_\phi (X_0+w_0), X_0+w_0\right) +\frac12 \left(X_1+w_1, X_1+w_1\right) + O(\beta^3)\\
 &=\mathcal{E}(\phi, 0)   + \frac12\bigg[(\LL_\phi w_0,  w_0) +   (w_1, w_1)\bigg] + \frac12\bigg[ (\LL_{\phi}X_0, X_0) +(X_1,X_1) \bigg]\notag\\
&  \quad \qquad + (\LL_{\phi} X_0 , w_0)
 + (X_1, w_1)   + O(\beta^3) .
\end{align*}
Since $\LL_\phi \rho_i = -k^2_i \rho_i$,    we get
\begin{align*}(\LL_\phi w_0,  w_0)=&-\sum_{i=1}^n (\mu^+_i+\mu^-_i)^2k_i^2  + (\LL_\phi \gamma_0,  \gamma_0),\\
(w_1, w_1) = &\sum_{i=1}^n (\mu^+_i-\mu^-_i)^2k_i^2+(\gamma_1,\gamma_1). \end{align*}
Combining the calculations above, we get (\ref{eq:Eu1}).
\end{proof}

Now we are ready to present the main idea of our paper, which is crucial to conclude that the set of initial data for which the solution scatters to an unstable steady state $(\phi,0)$ has a manifold structure, and hence is a ``thin set".
\begin{theorem}\label{th:nongeneric}
Let $V\in Y$ be a potential such that equation (\ref{eq:mainequation}) has only finitely many steady states, all of which are hyperbolic. \footnote{By a simple adaptation of the result in \cite{JiaLiuSchlagXu}, we know that such $V$ are dense in $Y$.} Suppose that  the finite energy
solution $\overrightarrow{U}(t)$ to equation (\ref{eq:mainequation})
scatters to an unstable excited state $(\phi,0)$. Let $\mathcal{M}$
be the local center-stable  manifold around $\overrightarrow{U}(0)$
and let $\epsilon_0,\,\epsilon_1$ be as defined in Theorem
\ref{th:localmanifold}. Then there exist $\epsilon$ with
$0<\epsilon<\epsilon_1<\epsilon_0$ and $\delta(\epsilon_1)\gg
\epsilon $, such that for any solution $u$ with finite energy
initial data $(u_0,u_1)\notin \mathcal{M}$ with
$$\left\|(u_0,u_1)-\overrightarrow{U}(0)\right\|_{\dot{H}^1\times
L^2}<\epsilon,$$ we can find $A>0$ such that for all $t\ge A$
\begin{equation}\label{eq:emitmoreenergy}
\int_{|x|\ge t-A} \left[\frac{|\nabla
u|^2}{2}+\frac{(\partial_tu)^2}{2}\right](x,t)\,dx\ge
\mathcal{E}(\overrightarrow{U}(t))-\mathcal{E}((\phi,0))+\delta.
\end{equation}
As a consequence, $\OR{u}(t)$ will  not scatter to $(\phi,0)$.
\end{theorem}

\begin{proof} We divide our proof into several steps.\\
\noindent
\textit{Step 1: Set up the parameters.}\\
By the local center-stable  manifold theorem of Section~\ref{sec:4}, the locally defined finite co-dimensional manifold $\mathcal{M}$ satisfies the property that any solution to equation (\ref{eq:mainequation}) with initial data on $\mathcal{M}$ scatters to $(\phi,0)$. Moreover, if a solution $\overrightarrow{u}(t)$ with initial data $(u_0,u_1)\in B_{\epsilon_1}(\overrightarrow{U}_0)$ satisfies
\begin{equation}\label{eq:eps 1}
\|\overrightarrow{u}(t)-\overrightarrow{U}(t)\|_{\dot{H}^1\times
L^2}<\epsilon_1 \,\,{\rm for\,\,all\,\,}t\ge 0,
\end{equation}
then $(u_0,u_1)\in \mathcal{M}$. Take
$\epsilon<\epsilon_1$ sufficiently small to be chosen below. Since
the solution $\overrightarrow{U}(t)$ scatters to $(\phi,0)$ as
$t\to\infty$, denoting  by $\overrightarrow{U}^L$  the scattered
linear wave, we have the property that
\begin{equation}\label{eq:convergencewithradiation}
\lim_{t\to\infty}\left\|\overrightarrow{U}(t)-\overrightarrow{U}^L(t)-(\phi,0)\right\|_{\dot{H}^1\times
L^2}=0.
\end{equation}
 This implies that
 \begin{equation}\mathcal{E}(\OR{U})= \mathcal{E}(\phi,0) + {\frac12}\left\|\OR{U}^L\right\|_{\HL}^2\label{eq:energysum}
\end{equation}
By (\ref{eq:convergencewithradiation}), the fact that
$\phi\in\dot{H}^1(\R^3)$ and $U^L\in L^5_tL^{10}_x( [0,\infty)\times
\R^3)$, for any small $ \delta_1>0$, we can first fix some large $L$ and then choose $T_1>L$ sufficiently large, such that for all $t\ge T_1$,
\begin{itemize}
\item (Free wave small in $L^6$ norm)
\begin{equation}
\|U^L(t)\|_{L^6(\R^3)} \leq \delta_1 \label{eq:smallL6}
\end{equation}
\item (Closeness of $\OR{U}$ to $\OR{U}^L+(\phi,0)$ and choice of the bounded region)
\begin{equation}\label{eq:decomp}
\|\overrightarrow{U}(t)-\overrightarrow{U}^L(t)-(\phi,0)\|_{\dot{H}^1\times
L^2} +\|\OR{U}^L(0)\|_{\HL(|x|\geq L)} +\|\phi\|_{\dot{H}^1(|x|\ge L)}\leq \delta_1;
\end{equation}
\item (Most energy of the free radiation is exterior)
\begin{equation}\label{eq:mostlyexterior}
\int_{|x|\ge t-T_1+L}|\nabla_{x,t}U^L|^2(t,x)\,dx\ge
\int_{\R^3}|\nabla_{t,x}U^L|^2(t,x)\,dx-\delta_1^2;
\end{equation}
\item (Control on the Strichartz norm of the radiation)
Let $$D:=\{(x,t):\,|x|\leq T_1+L-t,\,0\leq t\leq T_1\}.$$ Then we
have
\begin{equation}\label{eq:smallfreewave}
\|U^L\|_{L^5_tL^{10}_x(  (0,\infty)\times \R^3\backslash
D)}<\delta_1.
\end{equation}
\end{itemize}

\medskip

\noindent We remark that (\ref{eq:mostlyexterior}) is  a consequence of the strong Huygens principle and approximation by free waves with compactly supported initial data.   (\ref{eq:mostlyexterior}) ensures that $U^L$ can
essentially be taken as zero for our purposes inside the region
$|x|\leq t-T_1+L$ for $t\geq T_1$, which will be important to keep
in mind later, in order to  distinguish the second piece of
radiation.
By the continuous dependence of the solution to equation
(\ref{eq:mainequation}) on the initial data in
$\HL(\R^3)$ and by finite speed of propagation,  if we take
$\epsilon$ sufficiently small and initial data
$(u_0,u_1)\in\HL\backslash\mathcal{M}$ with
\begin{equation}\|(u_0,u_1)-\OR{U}(0)\|_{\HL}<\epsilon, \label{eq:closedata}\end{equation} then
\begin{equation}
\|\OR{u}(T_1)-\OR{U}(T_1)\|_{\HL}
\end{equation}
can be made sufficiently small. Hence, noting that
$\|V\|_{L^{5/4}_tL^{5/2}_x(|x|\ge |t|)}$ is finite, we can apply Lemma~\ref{lem:perturbation}
  to conclude that
\begin{equation}\label{eq:veryclose}
\|\overrightarrow{u}(t)-\OR{U}(t)\|_{\HL(|x|\ge t-T_1)}\leq
\delta_1, \,\,{\rm for\,\,all\,\,}t\ge T_1.
\end{equation}
(\ref{eq:veryclose}) means that we can effectively identify $\OR{u}$
with $\OR{U}$ in the exterior region $$|x|\ge t-T_1,\,t\ge T_1.$$
Hence by (\ref{eq:decomp}), we see that
\begin{equation}\label{eq:appbound}
\|\OR{u}(t)-\OR{U}^L(t)\|_{\HL(|x|\ge t-T_1+L)}\leq 3\delta_1,
\end{equation}
that is, we can also identify $\OR{u}$ with $\OR{U}^L$ in the
exterior region $|x|\ge t-T_1+L,\,t\ge T_1$.

In order to avoid any possibility of confusion due to the many
parameters, we remark that $ \delta_1$ and $\epsilon$ can be made as
small as we wish, and will be chosen later. $T_1,\,L$ depend on $ \delta_1$ and
$\OR{U}$ only. $\epsilon$ is a small free parameter below some threshold determined by $\delta_1$.
The key point for us is that
$\epsilon_1>0$ is fixed no matter how small $\epsilon$ is chosen, see~\eqref{eq:eps 1}.

Since $(u_0,u_1)\not\in \mathcal{M}$, there exists an exit time $T_2>0$ from the $\epsilon_1$ ball, i.e.,  such that
\begin{equation}\label{eq:expelled}
\|\OR{u}(T_2)-\OR{U}(T_2)\|_{\HL(\R^3)}= \epsilon_1.
\end{equation}
Note that the choices of $T_1$ and $L$ do  not depend on $\epsilon$. Therefore, by the continuous dependence of the solution on its initial data in
$\dot{H}^1\times L^2$, if we choose $\epsilon$ sufficiently small,
we can assume $T_2>2(L+T_1+1)$.\\

\noindent
\textit{Step 2: Analyze the size of discrete mode at time $T_2$.}

Let us analyze $\OR{u}(T_2)$ in more detail.
By the estimates (\ref{eq:decomp}) and (\ref{eq:expelled}) we can write
\begin{equation}
\OR{u}(T_2)=(\phi,0)+\OR{U}^L(T_2)+(w_0,w_1), \label{eq:uw}
\end{equation}
where $\OR{w}=(w_0,w_1)\in\HL$ satisfies
\begin{equation}
2\epsilon_1\ge\epsilon_1+\delta_1\ge\|\OR{w}\|_{\HL(\R^3)}\ge
\epsilon_1-\delta_1\ge\epsilon_1/2, \label{eq:smallw}
\end{equation}
if $\delta_1$ is chosen smaller than $\frac{\epsilon_1}{2}$.
We now  list several facts:

(i) From (\ref{eq:closedata}), we infer that
\begin{equation}\left|\mathcal{E}(\OR{u}) - \mathcal{E}(\OR{U})\right| \lesssim \epsilon \label{eq:close-energy}\end{equation}
 
(ii)  Rewrite the decompostion (\ref{eq:uw}) in the form 
\begin{align}
\OR{u}(T_2)= & (\phi,0)+(\Lambda_0,\Lambda_1), \label{eq:uLambda}\\
(\Lambda_0,\Lambda_1)= &\OR{U}^L(T_2)+ (w_0,w_1),\label{eq:Lambda}\\
(w_0,w_1)= &\sum_{i=1}^n \left[\mu^{+}_i (\rho_i, k_i\rho_i) +\mu_i^{-} (\rho_i, -k_i\rho_i)\right] + (\gamma_0, \gamma_1),  \label{de-o}
\end{align}
 with orthogonality conditions $\int   \rho_0\gamma_j\, dx= \int   \rho_1\gamma_j \, dx=0$, for all $1\leq j\leq n.$
(\ref{eq:smallw}) implies that
\begin{equation}\epsilon_1^2 \lesssim \|(\gamma_0, \gamma_1)\|_{\HL}^2 + \sum_{i=1}^n \bigg[  k_i^2(\mu^+_i -\mu^-_i)^2  + (\mu^+_i +\mu^-_i)^2\bigg].  \label{eq:lowerbound}  
\end{equation}

(iii) Expand the energy functional at $T_2$. Since $\Lambda_0= U^L(T_2) +w_0$, from (\ref{eq:smallL6}),  (\ref{eq:smallw}) and our a priori choice $\delta_1<\frac12\epsilon_1$, we have    $\|\Lambda_0\|_{L^6}\lesssim  \epsilon_1$. We now apply  Lemma~\ref{lem:expansion1} and obtain 
 \begin{align}
\mathcal{E}(\OR{u}(T_2)) =& \, \mathcal{E}(\phi, 0)  + \frac12\big[ (\LL_{\phi}{U^L}(T_2), U^L(T_2)) +({U^L_t}(T_2),U^L_t(T_2)) \big]  - \sum_{i=1}^n 2\mu^+_i\mu^-_ik_i^2  \notag\\
&\hspace{.2cm}+ \,\frac12\big[ (\LL_\phi \gamma_0,  \gamma_0)+ (\gamma_1,\gamma_1)\big] +\left(\LL_{\phi} {U^L}(T_2) , \,w_0\right)
 + \left(U^L_t(T_2), w_1\right)  + O(\epsilon_1^3).
 \label{eq:Eu2}
\end{align}
Note that using the $L^6$ estimate of $U^L$ in (\ref{eq:smallL6}), we further have
\begin{align}
&\frac12\bigg[ (\LL_{\phi}{U^L}(T_2), U^L(T_2)) +({U^L_t}(T_2),U^L_t(T_2)) \bigg] \notag\\
&= \frac12 \left\|\OR{U}^L(T_2)\right\|_{\HL}^2 + \bigg((-V+5\phi^4)U^L(T_2), U^L(T_2)\bigg)\notag \\
&=\frac12 \left\|\OR{U}^L(T_2)\right\|_{\HL}^2 + O(\delta_1^2)\label{eq:UL}
\end{align}
In view of   (\ref{eq:uw}),  (\ref{eq:appbound})  together with (\ref{eq:decomp}) implies that
\[\|w_0\|_{\dot{H}^1(|x|\geq T_2-T_1+L)} +\|w_1\|_{L^2(|x|\geq T_2-T_1+L)} \leq 4\delta_1.\]
Thus $(w_0,w_1)$ is small inside the region $\{|x|\geq T_2-T_1+L\}$, while
(\ref{eq:mostlyexterior}) implies that $\OR{U}^L$ is small inside the region $\{|x|<T_2-T_1+L\}$:
\[\|\OR{U}^L(T_2)\|_{\HL(|x| < T_2-T_1+L)}\leq \delta_1.\]
Hence we get that
\begin{equation}
 \bigg|(U^L_t(T_2), w_1) \bigg|  =  \left|\int_{\{|x|\geq T_2-T_1+L\}\cup \{|x| < T_2-T_1+L\}}
  U^L_t(x,T_2)\, w_1(x)  \,dx \right|\lesssim  \delta_1 \label{eq:ULtw1};
 \end{equation}
and that
 \begin{align}
& \bigg|\left(\LL_{\phi} {U^L}(T_2) , \,w_0\right)\bigg|\notag\\
&= \left| \int (-\Delta -V +5\phi^4)U^L(x,T_2)\, w_0(x)\, dx \right| \notag\\
  &=\left| \int_{\{|x|\geq T_2-T_1+L\}\cup \{|x| < T_2-T_1+L\}} \nabla U^L(T_2) \cdot\nabla w_0   \,dx    + \int ( -V +5\phi^4) U^L(T_2) \,w_0 \,dx\right|\notag\\
 & \lesssim \delta_1  \label{eq:LLULw}
 \end{align}
 Now let us combine  estimates (\ref{eq:UL}), (\ref{eq:ULtw1}), (\ref{eq:LLULw}) with  (\ref{eq:Eu2}),  noting (\ref{eq:energysum}), we deduce 
   \begin{align}
 \mathcal{E}(\OR{u}) = \mathcal{E}(\OR{U})
 +\frac12\bigg[ (\LL_\phi \gamma_0,  \gamma_0)+ (\gamma_1,\gamma_1)\bigg] - \sum_{i=1}^k 2\mu^+_i\mu^-_ik_i^2 +O(\delta_1 +\epsilon_1^3) \label{eq:Eu3}
\end{align}
(iv) Since $(\gamma_0,\,\gamma_1)$ is in the continuous spectrum and $\mathcal{L}_{\phi}$ has no zero eigenvalues or zero resonance, we have 
\[(\LL_\phi \gamma_0,  \gamma_0)+ (\gamma_1,\gamma_1) \gtrsim \|(\gamma_0,\gamma_1)\|^2_{\HL}\]
In combination with (\ref{eq:close-energy}), (\ref{eq:lowerbound}), and  (\ref{eq:Eu3}), this coercivity  yields 
\begin{align*}
&\mathcal{E}(\OR{u}) - \mathcal{E}(\OR{U}) + \sum_{i=1}^n 2\mu^+_i\mu^-_ik_i^2 \\
&=  \frac12\bigg[ (\LL_\phi \gamma_0,  \gamma_0)+ (\gamma_1,\gamma_1)\bigg] +\,O(\delta_1 +\epsilon_1^3)\\
&\gtrsim  \|(\gamma_0,\gamma_1)\|^2_{\HL} +O(\delta_1 +\epsilon_1^3)\\
&\gtrsim c\epsilon_1^2 -  \sum_{i=1}^n  \bigg[k_i^2(\mu^+_i -\mu^-_i)^2  + (\mu^+_i +\mu^-_i)^2\bigg] +O(\delta_1 +\epsilon_1^3).
\end{align*}
This implies that
\begin{align*}
& \sum_{i=1}^n 2\mu^+_i\mu^-_ik_i^2 + \frac{C}{2} \sum_{i=1}^n  k_i^2(\mu^+_i -\mu^-_i)^2  + \frac{C}{2}\sum_{i=1}^n (\mu^+_i +\mu^-_i)^2  \\
&\gtrsim  \epsilon_1^2   -\bigg|\mathcal{E}(\OR{U}) - \mathcal{E}(\OR{u})\bigg| - O(\delta_1+\epsilon_1^3)
\\
&\gtrsim  \epsilon_1^2 - C\epsilon - O(\delta_1+\epsilon_1^3).
\end{align*}
Since all the constants depend only   on $U$, we can choose $\delta_1, \epsilon\ll \epsilon_1^2 $ and   conclude that 
\begin{equation}
\sum_{i=1}^n |\mu^+_i|^2 +|\mu^-_i|^2\gtrsim  \epsilon_1^2 \label{eq:musum}
\end{equation}
Now we denote $|\mu_{max}|=\max\{|\mu_i^+|, |\mu_i^-|,  1\leq i\leq n\}$. We can find $1\leq i_0\leq n$  such that either $|\mu_{i_0}^+| =|\mu_{max}|$ or $|\mu^-_{i_0}|=|\mu_{max}|$.  
From
(\ref{eq:smallw}) and (\ref{eq:musum}), we get
\[2n|\mu_{max}|^2 \geq c\epsilon_1^2 \geq \frac{c}{4}\|\OR{w}\|_{\HL}^2,\]
hence $|\mu_{max}|\geq \sqrt{\frac{c}{8n}}\|\OR{w}\|_{\HL}$. The constant $c$ only depends on $ V$ and $\phi.$\\
 
\noindent
\textit{Step 3: Show the second emission of energy and finish the proof. }

\medskip

\textbf{Case 1: $|\mu_{max}|=|\mu_{i_0}^+|$. }    Consider the solution $\tilde{u}$ to equation (\ref{eq:mainequation}) with
 \[\OR{\tilde{u}}(T_2) = (\phi,0)+(w_0,\,w_1),\]
Take $\kappa=\sqrt{\frac{C}{8n}}$,  $K$ and $\epsilon_{\ast}$  corresponding to $R=0$ in Lemma~\ref{nonlinear}. Note that both parameters depend only on $V$. 
With these choices of parameters, we get $\varepsilon(\kappa,K,\epsilon_{\ast})>0$ from Lemma~\ref{lem:growth}. Shrinking $\epsilon_1$ if necessary, we can assume that $\epsilon_1<\varepsilon(\kappa,K,\epsilon_{\ast})$. We emphasize that none of these parameters depend on $\delta_1$ or $\epsilon$, which are free parameters at this point. This is very important of course, in order not to run into a circular argument. We also note that $T=T(\kappa,K,V)$ from Lemma~\ref{lem:growth} 
does not depend on $\delta_1$ or $\epsilon$. \\
We can now apply part (1) of Lemma~\ref{lem:growth} and part (2) of Lemma~\ref{nonlinear} to conclude that  
\begin{equation}\int_{|x|>t-(T_2+T)} |\partial_t \tilde{u}|^2(x,t)\, dx \geq c(\epsilon_1)>0, \hspace{1cm} \text{ for }t\geq T_2+T.\label{eq:sendenergy}\end{equation}
Denote $\Xi:= R^3\times[T_2,\,T_2+T]\bigcup \{(x,t):\,|x|>t-T_2-T,\,t\ge T_2+T\}$. Note that
$$\left(|V|+\phi^4\right)\chi_{\Xi}\in L^{\frac{5}{4}}_tL^{\frac{5}{2}}_x(\mathbb{R}^3\times \mathbb{R}),$$
and that $U^L+\tilde{u}$ is an approximate solution to (\ref{eq:mainequation}) with a right hand side $f$ with $\|f\|_{L^1_tL^2_x}\lesssim \delta_1$.
By bound (\ref{eq:smallfreewave}) and Lemma
\ref{lem:perturbation} (by treating $u$ as perturbation of $U^L+\tilde{u}$), if we choose $\delta_1$ sufficiently small, then  for $(x,t)\in\Xi$,
\begin{equation}\label{eq:anotherdecomp}
\OR{u}(t,x)=\OR{U}^L(x,t)+\OR{\tilde{u}}(x,t)+\OR{r}(x,t),\end{equation}
where the remainder term $\OR{r}$ satisfies
\begin{equation}\label{eq:errorsmall3}
\sup_{t\in \R}\|\OR{r}(t)\|_{\HL(\R^3)}\leq C\delta_1.
\end{equation}
The estimate (\ref{eq:appbound}), decomposition
(\ref{eq:anotherdecomp}) and the estimate on the remainder term
(\ref{eq:errorsmall3}) imply for $t\ge T_2$ (in particular, $t\geq T_2+T$)
\begin{equation}
\int_{|x|\ge t-T_1+L}|\nabla_{t,x}\tilde{u}|^2(x,t)\, dx\leq
C\delta_1,
\end{equation}
this combined with (\ref{eq:sendenergy}) implies that for $t\geq T_2+T$
\begin{equation}
\int_{t-T_1+L\ge |x|\ge t-(T_2+T)}|\nabla_{t,x}\tilde{u}|^2(x,t)\, dx\ge
c(\epsilon_1)-C\delta_1.
\end{equation}
Hence by estimating $\|\OR{u}(t)\|_{\HL}$ in different regions $\{|x|\ge t-T_1+L\}$ and $\{t-T_1+L\ge |x|\ge t-(T_2+T)\}$, we get that
\begin{align}
&\|\OR{u}\|^2_{\HL(|x|\ge t-(T_2+T))}\notag\\
&\geq  \left\|\OR{U}^L+\OR{\tilde{u}}\right\|^2_{\HL(|x|\ge t-(T_2+T)} -C_1\delta_1\notag\\
&\geq   \left\|\OR{U}^L\right\|^2_{\HL(|x|\ge t-T_1+L)} + c(\epsilon_1)-C_2\delta_1\notag\\
&\geq \left\|\OR{U}^L\right\|^2_{\HL(\R^3)}+c(\epsilon_1)-C_3\delta_1\ge \left\|\OR{U}^L\right\|^2_{\HL(\R^3)}+\frac12 c(\epsilon_1).\label{eq:uexe}
\end{align}
The last line holds when we choose $\delta_1$
sufficiently small.
(\ref{eq:emitmoreenergy}) is then proved with $A=T_2+T$ and
$\delta =\frac12c(\epsilon_1)>0$.

Now we prove that $u$ cannot scatter to $(\phi, 0)$  $t\rightarrow +\infty$. Suppose it does so with free radiation $\OR{u}^L$, i.e., 
\[\|\OR{u}(t)  -(\phi,0) -\OR{u}^L(t) \|_{\HL}\rightarrow 0,  \hspace{1cm} \text{ as  }t\rightarrow +\infty.\]
Then (\ref{eq:uexe}) implies that
\begin{equation}\label{eq:moreradiatedenergy}
\|\OR{u}^L(t)\|^2_{\HL(\R^3)}\ge \lim_{t\rightarrow \infty
}\|\OR{u}(t)\|^2_{\HL(|x|\geq t-A)}\ge
\|\OR{U}^L(t)\|^2_{\HL(\R^3)}+\delta.
\end{equation}
Note that
\[
\|\OR{U}^L\|^2_{\HL}=\mathcal{E}(\OR{U})-\mathcal{E}(\phi,0),\hspace{1cm}
\|\OR{u}^L\|^2_{\HL}=\mathcal{E}(\OR{u})-\mathcal{E}(\phi,0).\]
We have reached contradiction with (\ref{eq:moreradiatedenergy}) if
  $\|\OR{u}(0)-\OR{U}(0)\|_{\HL(\R^3)}$ is chosen small, and thus have proved the theorem in case~1.

\smallskip

  \textbf{Case 2:$|\mu_{max}|=|\mu_{i_0}^-|$.} We will show this is impossible if we take $\epsilon$ small enough.  In fact, again applying part (1) of Lemma~\ref{lem:growth} and part (2) of Lemma~\ref{nonlinear},
 consider the solution $\tilde{u}$ to equation (\ref{eq:mainequation}) with data \[\OR{\tilde{u}}(T_2) = (\phi,0)+(w_0,w_1).\]
  We can find a time $T>0$ such that
\begin{equation}\int_{|x|>|t-(T_2-T)|} |\partial_t \tilde{u}|^2(x,t)\, dx \geq c(\epsilon_1)>0, \hspace{1cm} \text{ for }t\leq T_2-T.\label{eq:sendenergy2}\end{equation}
By   taking $\epsilon$ sufficiently small, we can assume $T_2>2T$. Now
setting time $t=0$ in   (\ref{eq:sendenergy2}), we get $\|\OR{\tilde{u}}(0)\|_{\HL(|x|>\frac12 T_2)} >c(\epsilon_1).$

Introduce the set  
$$\Xi':=R^3\times[T_2-T,\,T_2]\bigcup \{(x,t):\,|x|>|t-(T_2-T)|,\,0\leq t\leq T_2-T\}.$$
In analogy  to case~1, if we choose $\delta_1>0$ sufficiently small, then for $(x,t)\in \Xi'$  we have
\begin{equation}\label{eq:anotherdecomp2}
\OR{u}(t,x)=\OR{U}^L(t,x)+\OR{\tilde{u}}(t,x)+\OR{r}(t,x),\end{equation}
with the remainder term $\OR{r}$ satisfying
\begin{equation}\label{eq:errorsmall}
\sup_{t\in \R}\|\OR{r}(t)\|_{\HL(\R^3)}\leq C\delta_1.
\end{equation}
From (\ref{eq:decomp}) and by our choice of $T_2$, i.e.,  $T_2>2(L+T_1+1)$ and $T_2>2T$, we have $\|\OR{U}^L(0)\|_{\HL(|x|>\frac12 T_2)}<\delta_1$ and
\begin{align*}
\|(u_0,u_1)\|_{\HL(|x|>\frac12 T_2)} \geq \|\nabla_{x,t}\tilde{u}(0)\|_{\HL(|x|>\frac12 T_2)} - C\delta_1 > c(\epsilon_1)-C\delta_1 >\frac{1}{2}c(\epsilon_1)
\end{align*}
The last inequality holds provided we take $\delta_1$ small enough.  This yields 
a contradiction to the finite energy of $\OR{U}(0)$ by choosing $\epsilon$ sufficiently small and $T_2$ sufficiently large. Hence case~2 does not arise and we are done. 
 \end{proof}

Next we prove the property of path connectedness. 
 \begin{theorem}\label{path} For any unstable excited state $(\phi, 0)$, the corresponding center-stable manifold $\mathcal{M}_{\phi}$ is path connected. 
 \end{theorem}

\begin{proof}
Given data $(u_0, u_1), (\tilde{u}_0,\tilde{u}_1)\in \mathcal{M}_{\vp}$, we denote the corresponding  solutions by $u, \tilde{u}$. Write $h=u-\vp, \ell=\tilde{u}-\vp$. 
Repeat step~1 and step~2  in  the proof of Theorem~\ref{th:localmanifold}. Then given any $\epsilon \ll 1$, we can find $T=T(\epsilon, u, \tilde{u})$, such that 
\begin{equation}\|h\|_{L^{6,2}_xL^{\infty}_t\cap L^{\infty}_xL^2_t ( \R^3\times[T,\infty))} <\epsilon,\hspace{0.1cm} \|\ell \|_{L^{6,2}_xL^{\infty}_t\cap L^{\infty}_xL^2_t ( \R^3\times[T,\infty))} <\epsilon.  \label{h-l-small}\end{equation}
Now we seek a function $w(\theta, t, x)$ of the form 
\begin{align}w(\theta, t, x) &=(1-\theta) u +\theta\tilde{u}+\eta \notag \\ &= \vp + (1-\theta) h +\theta \ell +  \sum_{i=1}^n  \lambda_i(\theta, t)\rho_i +\gamma(\theta, t,x) \label{w-formula}\end{align}
such that  for all $\theta\in [0,1]$, $\gamma(\theta, t,x)\perp \rho_i, i=1,\ldots, n$ and  $w(\theta,t,x)$ is a solution to equation (\ref{eq:mainequation}) that scatters to $\vp$.

For $\theta\in[0,1]$ fixed,   the equation satisfied by $\eta = \sum_{i=1}^n  \lambda_i(\theta, t)\rho_i +\gamma(\theta, t,x)$ is: 
\[\eta_{tt}-\Delta\eta -V(x)\eta + 5\vp^4\eta + N(\theta,h,\ell,\vp,\eta)=0,\]
where 
\[N(\theta,h,\ell,\vp,\eta) = (\vp + (1-\theta) h +\theta \ell +\eta )^5 -(1-\theta) (\vp+h)^5 -\theta (\vp +\ell)^5 - 5\vp^4 \eta.\]
Now  we  can repeat the stability condition (\ref{condition2}) and obtain  the reduced system of the form (\ref{system-2}).

In $N(\theta,h,\ell,\vp,\eta)$, the terms independent of $\eta$ are of the form 
\begin{align*}(\vp + (1-\theta) h +\theta \ell   )^5 -(1-\theta) (\vp+h)^5 -\theta (\vp +\ell)^5
=\sum_{i+j+k=5, i\leq 3} C(\theta, i, j, k) \vp^i h^j \ell^k. 
 \end{align*}
Notice that there are no  terms  $\vp^5$ or $\vp^4 h,  \phi^4\ell$. 
 
 Also, the linear term of $\eta$ in  $N(\theta,h,\ell,\vp,\eta)$ is
 \[5(\vp + (1-\theta) h +\theta \ell )^4 \eta   - 5\vp^4 \eta\]
hence  all linear terms involve a factor of $h$ or $\ell$.

Now we can repeat estimates (\ref{tilde-lambda})(\ref{tilde-gamma}),  then (\ref{N1-estimate}) for the linear term in $\eta$, (\ref{N2-estimate}) for higher order terms in $\eta$.  
We also have  the following estimate on terms independent of $\eta$
\[\left\|\sum_{i+j+k=5, i\leq 3} C(\theta, i, j, k)  \vp^i h^j \ell^k\right\|_{\st{1}{2}( [T,\infty)\times \R^3)}\lesssim \epsilon^2  \] 
 To sum up, using the $X$ norm defined in (\ref{define-X}), we conclude that 
 \begin{align*}\|(\lambda_1,\cdots, \lambda_n, \gamma)\|_{X([T,\infty))}\leq & \, L\epsilon^2 +  L\left( \sum_{i=1}^n |\lambda_i(\theta, T)| + \|(\gamma(\theta, T), \dot{\gamma}(\theta,T))\|_{\hl} \right)
\\ &+ L \epsilon \|(\lambda_1,\cdots, \lambda_n, \gamma)\|_{X([T,\infty))} + L \sum_{k=2}^5\|(\lambda_1,\cdots, \lambda_n, \gamma)\|_{X([T,\infty))}^k,\end{align*}
where $L>1$ is a constant only depending on the constants in the
reversed Strichartz estimates,  $\|\phi\|_{L^6(\R^3)}$ and
$\|\rho_i\|_{L^\infty_x\cap L^{6,2}_x}$. 

Moreover, in a similar fashion one sees that the difference of two solutions satisfies a similar estimate in which   the first two terms disappear.  Following step 3 of the proof for Theorem~\ref{th:localmanifold}, we  can use the contraction mapping principle and conclude that for  sufficiently small  data
\[ \sum_{i=1}^n |\lambda_i(\theta,T)| + \|(\gamma(\theta,T), \dot{\gamma}(\theta, T))\|_{\hl}\leq \delta\]
there is  a solution  $w$ as in (\ref{w-formula})   which solves (\ref{eq:mainequation}). We can also  check that $w$ scatters to $\phi$ as in step 4 of the proof for Theorem~\ref{th:localmanifold}.
 
In particular, let us take 
$\lambda_i(\theta,T)=\frac{1}{n}\delta \theta(1-\theta)$ and $\vec{\gamma}(\theta,T,x)=\vec{0}$. We claim that the corresponding solution $w(\theta, t, x)$  satisfies the following relation
\begin{equation} w(0,t,x)=u(t,x),\quad w(1,t,x)=\tilde{u}(t,x),\quad \text{ for all } t\in \R. \label{w-equal}\end{equation}
In fact, notice that $\lambda_i(0,T)=0, \vec{\gamma}(0,T,x)=\vec{0}$ implies $\lambda_i(0,t)=0, \vec{\gamma}(0,t,x)=\vec{0}$ for $t\geq T$, which further implies $w(0,t,x)=u(t,x), t\geq T$. Similarly we have $w(1,t,x)=\tilde{u}(t,x), t\geq T$. Then (\ref{w-equal}) follows from the  uniqueness of solutions to equation (\ref{eq:mainequation}).

Hence $\{\vec{w}(\theta, 0, x), \theta\in [0,1]\}$ is a path in $\mathcal{M}_{\phi}$ connecting the two data $(u_0, u_1), (\tilde{u}_0, \tilde{u}_1)$.
\end{proof}

Now we can finish the proof for our main theorem.
\begin{proof}[Proof of Theorem \ref{th:maintheoremintro}.]
We only consider the case in which $(\phi,0)$ is unstable; stable $(\phi,0)$   can be handled using standard
perturbation arguments and the reversed Strichartz estimates. We only note that due to the lack of local wellposedness of equation (\ref{eq:mainequation}) in the reverse Strichartz space $L^{6,2}_xL^{\infty}_t\cap L^{\infty}_xL^2_t$, we need to use the fact that
$$\lim_{T\to\infty}\|U-\phi\|_{L^{6,2}_xL^{\infty}_t\cap L^{\infty}_xL^2_t(\mathbb{R}^3\times[T,\infty))}=0,$$
if $U(t)$ scatters to $\phi$ as $t\to\infty$. This fact can be easily deduced by using the same argument as in Claim~\ref{claim1}.
In some small neighborhood of any
point $\OR{U}(0)$ on $\mathcal{M}_{\phi}$, $\mathcal{M}_{\phi}$
coincides with the local center-stable  manifold $\mathcal{M}$ of
codimension $n$ which we constructed in Section~\ref{sec:localmanifold}. By
Theorem \ref{th:nongeneric}, $\mathcal{M}_{\phi}$ is thus a global
manifold of co-dimension~$n$. The path-connectedness follows from Theorem~\ref{path}. 
\end{proof}

\bibliography{NLWbib}
\bibliographystyle{amsplain}

\end{document}